\documentclass[12pt, A4]{article}
\usepackage{amsmath,amssymb,amsfonts,mathrsfs,hyperref,microtype, amsthm}

\usepackage{geometry}
\usepackage[pdftex]{graphicx}
\usepackage{eufrak}
\usepackage{graphicx,color}
\usepackage[T1]{fontenc}
\usepackage{rotating}
\usepackage{booktabs}
\usepackage{mathtools}
\usepackage{float}
\usepackage{breqn}
\usepackage{graphicx}
\usepackage{caption}
\usepackage{multirow}
\usepackage{authblk}

\newtheorem{thm}{Theorem}[section]
\newtheorem{rem}{Remark}[section]
\newtheorem{definition}{Definition}[section]
\newtheorem{conj}{Conjecture}[section]
\newtheorem{lem}{Lemma}[section]
\newtheorem{prop}{Proposition}[section]
\newtheorem{cor}{Corollary}[section]

\numberwithin{equation}{section}

\def\HH{ \EuFrak H}
\def\N{{\rm I\kern-0.16em N}}
\def\R{{\rm I\kern-0.16em R}}
\def \E{{\rm I\kern-0.16em E}}
\def\P{{\rm I\kern-0.16em P}}
\def\F{{\rm I\kern-0.16em F}}
\def\B{{\rm I\kern-0.16em B}}
\def\C{{\rm I\kern-0.46em C}}
\def\G{{\rm I\kern-0.50em G}}
\def\tr{{\mbox{Trace}}}

\newcommand{\LG}{\mathrm{L}}

\newcommand{\LL}{\mathfrak{L}}

\newcommand{\Z}{\mathbb{Z}}
\numberwithin{equation}{section}

\font\eka=cmex10
\usepackage{ae}

\def\ind{\mathrel{\hbox{\rlap{%
\hbox to 7.5pt{\hrulefill}}\raise6.6pt\hbox{\eka\char'167}}}}
\begin{document}

\title{\Large{\bf New moments criteria for convergence towards normal product/tetilla laws }}

\author[1]{Ehsan Azmoodeh}
\author[2]{Dario Gasbarra}
\affil[1]{Ruhr University Bochum}
\affil[2]{University of Helsinki}
\date{ }                     
\setcounter{Maxaffil}{0}
\renewcommand\Affilfont{\itshape\small}

\maketitle

\abstract In the framework of classical probability, we consider the normal product distribution $F_\infty \sim N_1 \times N_2$ where $N_1, N_2$ are two 
independent standard normal random variable, and in the setting of free probability, $F_\infty \sim \left( S_1 S_2 + S_2 S_1 \right)/\sqrt{2}$ known
as {\it tetilla law} \cite{d-n}, where $S_1, S_2$ are  freely independent normalized semicircular random variables. 
We provide novel characterization of $F_\infty$ within the second Wiener (Wigner) chaos. More precisely, we show that for any generic element $F$ 
in the second Wiener (Wigner) chaos with variance one the laws of $F$ and $F_\infty$ match if and only if $\mu_4 (F)= 9 \, (\mbox{resp. }\varphi(F^4)=2.5)$, and $\mu_{2r}(F)= ((2r-1)!!)^2 \, (\mbox{resp. }\varphi(F^{2r})=\varphi(F^{2r}_\infty ))$ for some $r \ge 3$,
where $\mu_r (F)$ stands for the $r$th moment of the random variable $F$, and $\varphi$ is the relevant tracial state. We use our moments characterization to study the non central limit theorems within the second Wiener (Wigner) chaos
and the target random variable $F_\infty$. Our results generalize the findings in Nourdin \& Poly  \cite{n-p-2w}, 
Azmoodeh, et. al \cite{a-p-p} in the classical probability, and of Deya \& Nourdin \cite{d-n} in the free probability setting.


 \begin{small}
 \tableofcontents
 \end{small}
\section{Introduction and main results}
In the landmark article \cite{n-p-4m}, Nualart \& Peccati established an impressive result known nowadays as the {\it fourth moment theorem} providing a drastically simple criterion in terms of the fourth moment for the normal approximation within a fixed Wiener chaos. A few years after, their findings create a fertile line of research 
and it is culminating in the popular article \cite{n-p-ptrf}, introducing the so called  {\it Malliavin-Stein} approach, 
an elegant combination of two probabilistic techniques in order to quantify the probability distance of a square integrable
Wiener functional from that of a normal distribution. The reader may consult the constantly updated web resource \cite{WWW} for a huge amount of applications and generalizations of the forthcoming theorems. We refer the reader to \cite{n-p-book,n-book}, as well as Sections. \ref{sec:useful-2w},\ref{sec:fp} below for any unexplained notion
evoked in the present section. The following two results provide a full characterization of the normal approximation on Wiener chaos in terms of (even) moments.

\begin{thm}\label{T:NPEC} {\rm (Fourth moment theorem \cite{ n-p-4m,n-l})} Let $N\sim \mathscr{N}(0,1)$. Fix $p\geq 2$ and let $\{F_n\}_{n\ge 1}$ be a sequence of 
	random variables in the Wiener chaos of order $p$ such that $\E(F^2_n) = 1$ for all $n \geq 1$. Then, as $n\to \infty$, the following statements are equivalent. 
	\begin{enumerate}
		\item[ \rm (i) ] $F_n$ converges in distribution to $N$;
		\item[ \rm (ii) ] $\E(F^4_n) \to \E( N^4) = 3$.
	\end{enumerate}
	
	\smallskip
	
\end{thm}

Theorem \ref{T:NPEC} has been extended in several important directions, for instance to {\it Gamma approximation} \cite{n-p-noncentral}, and their free probability counterparts \cite{k-n-p-s,n-p-34moments} just to mention a few. In another breakthrough article \cite{ledoux}, Ledoux introduced the novel technique of {\it Markov triplet} to prove fourth moment theorems for a sequence of eigenfunctions of a diffusive Markov operator. His approach has been studied further in \cite{ledoux,a-c-p}, and lead to the following remarkable generalization of the fourth moment theorem. 

\begin{thm}\label{thm:hm-normal} {\rm (Even moment theorem \cite{a-m-m-p})}
	Let $N\sim \mathscr{N}(0,1)$. Fix $p\geq 2$ and let $\{F_n\}_{n\ge 1}$ be a sequence of 
	random variables in the Wiener chaos of order $p$ such that $\E(F^2_n) = 1$ for all $n \geq 1$. Then, as $n \to \infty$, the following asymptotic assertions are equivalent.
	\begin{enumerate}
		\item[ \rm (i)] $F_n$ converges in distribution to $N$;
		\item[ \rm (ii)] $\E( F^{2r}_n) \to \E (N^{2r})=(2r-1)!!$ for some $r \ge 2$.
	\end{enumerate}
\end{thm}
In this article, we consider two important probability distributions namely the so called {\it normal product} and the {\it tetilla} laws,  corresponding   to
random variables in the second Wiener and Wigner chaos respectively. 
Their
canonical representation $(\ref{spectral:representation:classical})$ in item $1$, Proposition \ref{second-property}, 
is spectrally symmetric with two non zero coefficients $\lambda_{f_\infty,\pm 1}=\lambda_{\pm 1} =\pm 1/\sqrt{ 2}$. 
These two probability distributions appear naturally in several contexts as discussed
for example in the recent articles \cite{b-t,p-r,tetilla-application}. 
Our principal aim is to study  higher (even) moments characterizations of the aforementioned target distributions in analogy with Theorem \ref{thm:hm-normal}  for the normal distribution.

We start with the normal product distribution in the framework of classical probability. In below $d_{W_2}$ denote the {\it Wasserstein$-2$ distance}, see Section \ref{sec:main-results} for precise definition.  Our main finding reads as follows. 

\begin{thm}
Let $F_\infty \sim N_1 \times N_2$ where $N_1, N_2 \sim \mathscr{N}(0,1)$ are two independent standard normal random variables. Assume that $\{F_n\}_{n \ge 1}$ is a sequence of the random elements in the second Wiener chaos such that $\E(F^2_n)=1$ for all $n \ge 1$. Then, as $n \to \infty$, the following asymptotic assertions are equivalent.
\begin{description}
	\item[(I)] $d_{W_2}(F_n,F_\infty) \to 0$.
\item[(II)] sequence $F_n  \to F_\infty$ in distribution.
\item[(III)]  as $n \to \infty$, 
\begin{enumerate}
\item $\mu_4 (F_n) \to 9$.
\item $\mu_{2r}(F_n) \to \big( (2r-1)!!\big)^2$ for {\bf some} $r \ge 3$.
\end{enumerate}
\end{description}
\end{thm}

In the free probability setting, the tetilla law $F_\infty \sim \left( S_1 S_2 + S_2S_1 \right)/\sqrt{2}$ plays the same role as the normal product probability distribution. Here $S_1$, and $S_2$ stand for two freely independent {\it semicircular} random variables. Recently, in \cite{d-n}, Deya \& Nourdin proved that for a sequence $\{F_n\}_{n \ge 1}$ of standardized random elements in a fixed Wigner chaos of order $p \ge 2$, the sequence $F_n \to F_\infty$ in distribution if and only if $\varphi(F^4_n) \to \varphi(F^4_\infty) = 2.5$, and  $\varphi(F^6_n) \to \varphi(F^6_\infty) = 8.25$. Our next result generalizes the main result in \cite{d-n}.

\begin{thm}
Let $\{F_n\}_{n \ge 1}$ be a sequence of non-commutative random variables in the second Wigner chaos such that $\varphi(F^2_n)=1$ for all $n \ge 1$. Assume that $F_\infty$ distributed as normalized tetilla distribution as explained in above. 
Then, as $n \to \infty$, the following asymptotic assertions are equivalent.
\begin{description}
	\item[(I)] $d_{W_2}(F_n,F_\infty) \to 0$.
\item[(II)] sequence $F_n  \to F_\infty$ is distribution.
\item[(III)]  as $n \to \infty$, 
\begin{enumerate}
\item $\varphi (F_n^4) \to 5/2$.
\item $\varphi(F_n^{2r} ) \to  \varphi(F_{\infty}^{2r} )     $ for {\bf some} $r \ge 3$.
\end{enumerate}
\end{description}
\end{thm}



The following remarks are in order. Let $F$ be a normalized ($\E(F^2)=1$) random element in an arbitrary Wiener chaos of order $p \ge 2$.
\begin{itemize}
\item[--] Let $\LG$ and $\Gamma$ stand for the Ornstein-Uhlenbeck and the associated carr\'e-du-champ operators, see \cite{ledoux,a-c-p, a-m-m-p} for definitions. It is a well known fact in the Malliavin-Stein approach that in the total variation distance $$d_{\text{TV}}(F, \mathscr{N}(0,1)) \le_{C_p} \sqrt{\text{Var}\left( \Gamma(F) \right)}.$$  The crucial fact that the quantity $\Gamma(F)=P(L)(Q(F))$ for some specific polynomials $P$, and $Q$ allows one to estimate the later variance quantity by higher even moments by relying on only the knowledge of the spectrum $\bf{sp}(\LG)= - \N$.
\item[--] In the classical probability setting, the normal product random variable $F_\infty \sim N_1 \times N_2$ belongs
to the so-called {\it Variance--Gamma} class. The Malliavin--Stein technique implemented in \cite{e-t}(see also \cite{a-c-p} for a general class of target distributions) reveals that 
 $$d_{W_1}(F, F_\infty) \le_C \Big\{\sqrt{  \text{Var}\left( \Gamma_2 (F) - F \right) } + \vert  \kappa_3 (F) \vert \Big\}.$$ Here $d_{W_1}$ denote the Wasserstein-1 metric, and noticing that $\kappa_3(F_\infty)=0$.  The variance quantity appearing in the RHS of the above estimate contains the iterated double Gamma operator $\Gamma_2$. It readily can be shown that the complicated Gamma operator $\Gamma_2$ cannot be written in the form of a polynomial in $\LG$ operator in order to mimic the methodology as explained in the previous item for the normal approximation.  
\item[--] In order to successfully turn around the obstacles as explained, instead we consider the case of second Wiener/Wigner chaos, and take advantages of the spectral representations
$(\ref{spectral:representation:classical}, \ref{2nd:wigner:chaos:representation})$. To achieve our main results we carry out an elegant moment/cumulant analysis at Sections \ref{sec:cp}, \ref{sec:fp}. As a by product, we obtain interesting moments/cumulants inequalities such as Propositions \ref{prop:sym-case} and \ref{prop:sym-case-non}.
\end{itemize}

\subsection{Plan}

The paper is organized as follows. Section \ref{sec:useful-2w} contains some preliminary material
including basic facts on second Wiener chaos and iterated Gamma operators. Sections \ref{sec:cp}, and \ref{sec:fp} are
devoted to characterizations of $N_1\times N_2$ and the tetilla laws within the second Wiener/Wigner chaos respectively. Section \ref{sec:main-results} contains our main results on Wasserstein-2 convergence
towards $N_1 \times N_2$ and the tetilla laws in terms of higher even moments criteria.
Finally the paper ends with Section  \ref{sec:conjecture}, including a conjecture arising from our study.

\subsection{ Cumulants }
The notion of cumulant will be crucial throughout the paper. We refer the reader to the monograph \cite{p-t-book} for an exhaustive discussion.

\begin{definition}[Cumulants]\label{D : cum}{\rm Let $F$ be a real-valued random variable such that $\E|F|^m<\infty$ for some integer
$m\geq 1$. 
The $F$-cumulants  $\kappa_{\ell}(F)$, $\ell=1,\dots,m$ are defined by the relations
\begin{align*}
 \E\bigl(F^{\ell} \bigr) = \sum_{\pi\in \Pi_{\ell} } \prod_{ A\in \pi} \kappa_{|A|}(F), \quad \ell=1,\dots,m,
\end{align*}
where we sum  over the partitions of $\{ 1,2,\dots,\ell\}$, and $\vert\pi\vert$ is the number of subsets of the partition $\pi$.
M\"obius inversion  on the partitions lattice gives the explicit definition  
\begin{equation}
\kappa_{\ell} (F) = \sum_{\pi\in \Pi_{\ell} }(\vert \pi \vert-1)! (-1)^{\vert\pi\vert-1} \prod_{ A\in \pi} \E\bigl( F^{|A|} \bigr).
\end{equation}
}
\end{definition}
\begin{rem}{\rm When $\E(F)=0$, then the first six cumulants of $F$ are the following: $\kappa_1(F) = \E(F)=0$, $\kappa_2(F) = \E(F^2)= {\rm Var}(F)$,
$\kappa_3(F) = \E( F^3)$, $ \kappa_4(F) = \E(F^4) - 3\E(F^2)^2$, and $$\kappa_6(F)=\E(F^6)- 15 \E(F^2) \E(F^4)-10 \E(F^3)^2 +30 \E(F^2)^3.$$ Hence, $\E(F^6)=\kappa_6(F)+15 \kappa_2(F) \kappa_4(F) + 10 \kappa^2_3(F) + 15 \kappa^3_2(F)$.
}
\end{rem}

\section{Useful facts about the second Wiener chaos}\label{sec:useful-2w}
We recall some relevant information about the elements in the second Wiener chaos. For a comprehensive treatment, we refer the reader to \cite[Chapter 2]{n-p-book}.  Consider an isonormal process $W=\{W(h); \ h \in \HH\}$ over a separable Hilbert space $\HH$. Recall that the second Wiener chaos $\mathscr{H}_2$ associated to the isonormal process $W$ consists of those random variables having the general form $F=I_2(f)$, with $f \in  \HH^{\odot 2}$. Notice that, if $f=h\otimes h$, where $h \in \HH$ is such that $\Vert h \Vert_{\HH}=1$, then using the multiplication formula (see \cite{n-p-book}), one has 
$I_2(f)=W(h)^2 -1  \stackrel{\text{law}}{=} N^2 -1$, where $N \sim \mathscr{N}(0,1)$. To any kernel $f \in \HH^{\odot 2}$, we associate the following \textit{Hilbert-Schmidt} operator
\begin{equation*}
A_f : \HH \mapsto \HH; \quad g \mapsto f\otimes_1 g. 
\end{equation*}
It is also convenient to introduce the sequence of auxiliary kernels
\begin{equation}
\left\{ f\otimes _{1}^{\left( p\right) }f:p\geq 1\right\} \subset \mathfrak{H%
}^{\odot 2}  \label{kern1}
\end{equation}%
defined as follows: $f\otimes _{1}^{\left( 1\right) }f=f$, and, for $p\geq 2$,%
\begin{equation}
f\otimes _{1}^{\left( p\right) }f=\left( f\otimes _{1}^{\left(
p-1\right) }f\right) \otimes _{1}f\text{.}  \label{kern2}
\end{equation}
In particular,
$f\otimes _{1}^{\left( 2\right) }f=f\otimes _1 f$. Finally, we write $\{\lambda_{f,j}\}_{j \ge 1}$ and $\{e_{f,j}\}_{j \ge 1}$, respectively, to indicate the (not necessarily distinct) eigenvalues of $A_f$ and the corresponding eigenvectors.

\begin{prop}[See e.g. Section 2.7.4 in \cite{n-p-book}] \label{second-property}
Fix $F=I_2(f)$ with $f \in \HH^{\odot 2}$.
\begin{enumerate}
 \item The following equality holds:
 \begin{align}\label{spectral:representation:classical}
 F=\sum_{z \in \Z} \lambda_{f,z}  \frac{ \bigl( N_{z}^2 - 1 \bigr ) }{\sqrt 2 },\end{align} 
 where $\{N_z \}_{ z \in \Z}$ is a sequence of i.i.d. $\mathscr{N}(0,1)$ random variables that are elements of the isonormal process $W$, and the series converges in $L^2$ and almost surely.
 
 \item For any $i \ge 2$,
 \begin{equation}\label{classical:spectral:representation}
  \kappa_r(F)= 2^{\frac{r}{2}-1}(r-1)! \sum_{ z \in \Z } \lambda_{f,z}^r = 2^{\frac{r}{2} -1}(r-1)! \times \langle f \otimes^{(r-1)}_{1}f ,f \rangle_{\HH^{\otimes 2}}.
 \end{equation}
\item The law of the random variable $F$ is completely determined by its moments or equivalently by its cumulants.
\end{enumerate}
\end{prop}

Consider a generic element $F \in \mathscr{H}_2$. Since the reordering the coefficients in the representation 
$F$ at item $1$ in Proposition $(\ref{second-property})$ does not change the distribution of $F$, we will see that 
it is very useful with the convention $\lambda_{f,z}=\lambda_z$ to consider ordered coefficients with $\lambda_0=0$ and
\begin{align} \label{eigenvalue:ordering}
-\Vert F \Vert_{L^2} \le \lambda_{-1 } \le \lambda_{-2} \le \dots \le \lambda_{-n}  \le \dots \le 0 \le \dots \le \lambda_n \le \dots \le \lambda_2 \le \lambda_1 \le \Vert F \Vert_{L^2} ,
\end{align}
and to separate positive and negative  coefficients in the decomposition 
\begin{align*}F=  F_+ - F_{-}, \end{align*} where $F_{\pm}$ are the independent centered
Generalized Gamma Convolutions (GCC)
\begin{align*}
  F_{\pm} =\sum_{\ell \in \N } |\lambda_{\pm \ell }| \frac{ \bigl( N_{\pm \ell}^2 - 1 \bigr ) }{\sqrt 2 } , 
\end{align*}
We assume that $F$ is normalized with  $ \E(F^2)=\E(F_+^2)+ \E(F_-^2)=1$.\\

Our aim is now to provide an explicit representation of cumulants in terms of Malliavin operators. To this end, it is convenient to introduce the following definition (see e.g. \cite[Chapter 8]{n-p-book} for a full multidimensional version).

\begin{definition}\label{Def : Gamma} {\rm Let $F\in \mathbb{D}^{\infty}$. The sequence of random variables $\{\Gamma_i(F)\}_{i\geq 0}\subset
\mathbb{D}^\infty$ is recursively defined as follows. Set $\Gamma_0(F) = F$
and, for every $i\geq 1$, \[\Gamma_{i}(F) = \langle DF,-DL^{-1}\Gamma_{i-1}(F)\rangle_{\HH}.
\]}
\end{definition}

The following statement explicitly connects the expectation of the random variables $\Gamma_i(F)$ to the cumulants of $F$. 

\begin{prop}\label{p:cumgamma}{\rm (See  Chapter 8 in \cite{n-p-book})} Let $F\in \mathbb{D}^\infty$. Then $F$ has finite moments of every order, and the following relation holds for every $r \geq 0$:
\begin{equation}\label{e:cumgamma}
\kappa_{r+1}(F) = r! \E[\Gamma_r(F)].
\end{equation}
\end{prop}
 \begin{prop} 
The law of 
\begin{align*}
N_1 \times  N_2 \stackrel{law}{=}\frac{ N_1^2- N_2^2} 2 
 =\frac{ ( N_1^2- 1)- (N_2^2-1) } 2 \end{align*}
 is characterized in the second chaos by the coefficients $\lambda_{\pm 1}=\pm 1/\sqrt 2$, and $\lambda_{z}=0$ for $|z| \neq 1$
 \end{prop}
 
 \begin{proof} It is a direct consequence of the shape of the characteristic functions of elements in the second Wiener chaos, see \cite[page $44$]{n-p-book}.
 \end{proof}
 
 \section{Classical probability and $N_1 \times N_2$ law}\label{sec:cp}
\subsection{Characterization of $N_1 \times N_2$ within the second Wiener chaos}
We start with the following general characterization of $N_1 \times N_2$ law inside the second Wiener chaos in terms of iterated Gamma operators. 
\begin{prop}\label{prop:Gamma_nm}
 Assume $F=I_2 (f)$ be an element in the second Wiener chaos with $\E(F^2)=2\Vert f\Vert^2 =1$. Then the following assertions are equivalent.
\begin{description}
\item[(I)] the laws of $F$ and $N_1 \times N_2$ coincide, i.e $F \sim N_1 \times N_2$.
\item[(II)] for some $m \neq n \in \N$ such that $n+m \in 2\N$
\begin{itemize}
\item[(1)] $\Delta_{n,m}(F):= \text{Var} \left( \Gamma_{n-1} (F) - \Gamma_{m-1} (F) \right) =0$,
\item[(2)] $\kappa_r (F)=0$ for some {\bf odd} $r \ge 3$.
\end{itemize}
\end{description}
\end{prop}
\begin{proof}
Since $F=I_2(f)=\sum_{ i \in \Z} \lambda_i (N^2_i -1)/\sqrt{2}$ belongs to the second Wiener chaos, we have the nice representation (see \cite[ relation (3.7), Lemma 3.1]{a-p-p})
$$ \Gamma_{n-1} (F) - \E(\Gamma_{n-1}(F)) - \big( \Gamma_{m-1} (F) - \E(\Gamma_{m-1}(F)) \big) = I_2 \left( 2^{n/2 -1} f \otimes^{(n)}_1 f - 2^{m/2-1} f \otimes^{(m)}_1 f \right).$$ Therefore,
\begin{equation}\label{eq:1}
\begin{split}
\Delta_{n,m}(F) & = \E \left (I_2 ( 2^{n/2-1} f \otimes^{(n)}_1 f - 2^{m/2 -1} f \otimes^{(m)}_1 f ) \right)^2\\
&= 2 \Big \Vert 2^{n/2 -1} f \otimes^{(n)}_1 f - 2^{m/2 -1} f \otimes^{(m)}_1 f \Big \Vert^2\\
&= \frac  1 2 \sum_{i \ge 1} \left( 2^{\frac{n}{2}} \lambda^n_i - 2^{\frac{m}{2}} \lambda^m_i \right)^2\\
&= \frac{\kappa_{2n}(F)}{(2n-1)!} - 2 \frac{\kappa_{n+m}(F)}{(n+m-1)!} + \frac{\kappa_{2m}(F)}{(2m-1)!}.
\end{split}
\end{equation}

Now assume that there are $m < n$ with $(n+m) \in 2\N$ and $\Delta_{n,m}(F)=0$. Relation $(\ref{eq:1})$ implies
that            $\lambda^2_i =\{ 0, \frac{1}{2} \}$
for all $i \ge 1$. Combining this with the second moment assumption $\E(F^2)=1$ we deduce that there are exactly 
two non zero coefficients with  $\lambda_i^2=\lambda_j^2=\frac{1}{2}$ and $i\ne j$. If furthermore for some odd $r \ge 3$ we have $\kappa_r (F)=0$
then necessarily $\lambda_i=\pm \frac {1} {\sqrt 2}, \lambda_j= \mp \frac {1} {\sqrt  2}$ with opposite signs.
Hence $F \sim N_1 \times N_2$. The other direction is obvious, because for $F \sim N_1 \times N_2$, we have $\kappa_{2r}(F)= (2r-1)!$ for $r \ge 1$, and $F$ is a symmetric distribution and as a result all the odd cumulants must be zero.
\end{proof}

\begin{rem} { \rm Let $F=I_2(f)$ be a general element in the second Wiener chaos such that $\E[F^2]= 2\Vert f \Vert^2 =1$.  Then proof of Proposition \ref{prop:Gamma_nm} reveals
that condition $\Delta_{n,m}(F)=0$ for some $ n \neq m \in \N$ with $n +m \in 2\N$ implies that distribution of the random variable $F$ belongs to the set
of three possible probability distributions 
$$ \Big\{ \frac{ ( N^2_1 -N_2^2) }{2} , \pm\frac{ (N_1^2+N^2_2 -2)}{2}\Big\}.$$ Hence, in order to distribution $F$ lies exactly on the favorite target random variable $F_\infty \sim N_1 \times N_2$, 
one needs to fix $\kappa_r(F)=0$ for at least one (and therefore for every) odd $r \ge 3$. Moreover, one has to note that outside of the second Wiener chaos those conditions stated in Proposition \ref{prop:Gamma_nm}
do not characterize the law of random variable $F_\infty$. A  standard Gaussian random variable is a simple counterexample.

}
\end{rem}

\begin{cor}\label{cor:chain}
Let $F=I_2(f)$ be a general element in the second Wiener chaos such that $\E(F^2)= 2\Vert f \Vert^2 =1$. Let $n, m \in \N$. 
Then the following chain of the estimates take place,
$$\cdots \le_C \Delta_{n,m}(F) \le_C \Delta_{n-1,m-1}(F) \le_C \cdots \le_C \Delta_{3,1}(F)=\text{Var}\left(\Gamma_2 (F) - F \right),$$
where the quantity $$\Delta_{n,m}(F)= \frac{\kappa_{2n}(F)}{(2n-1)!} - 2 \frac{\kappa_{n+m}(F)}{(n+m -1)!} +\frac{\kappa_{2m}(F)}{(2m-1)!}$$
is given at item $(1)$ in Proposition \ref{prop:Gamma_nm}. Moreover, if $\Delta_{n,m}(F)=0$ for {\bf some} $n \neq m \in \N$ with $(n +m) \in 2\N$,
then $\Delta_{3,1}(F)=0$, and therefore $\Delta_{n,m}(F)=0$ for {\bf all} $n \neq m \in \N$ with $(n+m) \in 2\N$.  In particular, for $F \sim N_1 \times N_2$, we have $\Delta_{n,m}(F)=0$ for  all $n \neq m \in \N$ with $n+m \in 2\N$. 
\end{cor}

\begin{proof}
Let assume the nontrivial case $n \neq m$. Then relation $(\ref{eq:1})$ yields that 
\begin{equation}\label{eq:2}
\begin{split}
\Delta_{n,m}(F) & = 2 \Big \Vert  2^{n/2 -1} f \otimes^{(n)}_1 f - 2^{m/2 -1} f \otimes^{(m)}_1 f  \Big \Vert ^2\\
&= 8 \Big \Vert f \otimes_1 g \Vert^2, \qquad g=  2^{n/2 -2} f \otimes^{(n-1)}_1 f - 2^{m/2 -2} f \otimes^{(m-1)}_1 f \\
& \le 8 \Vert f \Vert^2 \, \Vert g \Vert^2\\
& = 2 \Delta_{n-1,m-1}(F).
\end{split}
\end{equation}
Now assume that we have $\Delta_{n,m}(F)=0$ for some $n \neq m \in \N$ with $(n+m) \in 2\N$.  Then proof of Proposition \ref{prop:Gamma_nm} tells us that all 
the nonzero coefficients must satisfy in $\lambda^2_i = \frac{1}{2}$, and therefore $\Delta_{3,1}(F)=0$ which implies that $\Delta_{n,m}(F)=0$ for all $n \neq m \in \N$ with $(n+m) \in 2\N$.
\end{proof}
The following result aims to provide some variance calculus of the iterated Gamma random variables.

\begin{prop}\label{prop:variance-calculus}
Let $F=I_2(f)$ be a random variable in the second Wiener chaos with $\E(F^2)=1$. Then for $r \in \N$ there exists a constant $C=C_r$ such that the following variance estimates take place.
\begin{align*}
\text{Var}^{\, 2} \left( \Gamma_{r+1}(F) - \Gamma_{r-1}(F)\right) & \le_C \text{Var} \left( \Gamma_r(F) -\Gamma_{r-2}(F) \right) \times \text{Var} \left( \Gamma_{r+2}(F) - \Gamma_{r}(F)\right), \quad r \ge 2\\
\text{Var}^{\, 2r} \left( \Gamma_{3}(F) - \Gamma_{1}(F)\right) & \le_C \text{Var}^{\, 2r-1} \left( \Gamma_{2}(F) - F \right) \times \text{Var} \left( \Gamma_{2r+2}(F) - \Gamma_{2r}(F) \right), \quad r \ge 1.
\end{align*}
In particular case, we obtain 
$$\text{Var}^{\, 2} \left( \Gamma_{3}(F) - \Gamma_{1}(F)\right)  \le_C \text{Var} \left( \Gamma_2(F) - F \right) \times \text{Var} \left( \Gamma_{4}(F) - \Gamma_{2}(F)\right).$$
\end{prop}

\begin{proof}
Denote $A_f : \HH \to \HH$ defined as $g \mapsto \langle f , g \rangle_\HH$ the associated Hilbert-Schmidt operator to the kernel $f$. It is well known that for $r \ge 2$ (see for example \cite{n-p-book}) 
$$\kappa_r(F)= 2^{\frac{r}{2}-1}(r-1)! \text{Tr}(A^{r}_f)$$ where $\text{Tr}(A^r_f)$ stands for the trace of the $r$th power of $A_f$. Using relation $(\ref{eq:1})$ together with some direct computations one can get that for $r \ge 2$, $$\text{Var} \left( \Gamma_r (F) - \Gamma_{r-2}(F) \right) = 2^{2r-3} \, \text{Tr} \big( (2^2 A^{r+1}_f - A^{r-1}_f)^2 \big).$$ Now, the first variance estimate is an application of \cite[Corollary 1]{D-trace-1} with $P = (A^{r+2}_f - A^r_f )^2, C = A^2_f$ , and the second variance estimate can be deduced from \cite[Corollary 1]{D-trace-2} with $P=(2^2 A^3_f - A_f)^2$ and the convex function $f(x)=x^{2r}$.
\end{proof}

\subsection{Case $\E(F^4) \ge 9$}
\begin{prop}\label{prop:cumulants-estimate}
Let $F$ be a general element in the second Wiener chaos, and $F_\infty \sim N_1 \times N_2$. Then, for every $ r \ge 2$, 
\begin{equation}\label{eq:even-cumulants-1}
\frac{\kappa_{2r}(F)}{(2r-1)! \kappa_2(F) } -1 \ge (r-1) \Big\{ \frac{\kappa_4(F)}{3! \kappa_2(F)} - 1 \Big\}.
\end{equation}
When $\kappa_2 (F) =1$, we have for $r \ge 2$
\begin{enumerate}
\item $\kappa_4 (F) - \kappa_4(F_\infty) \le \frac{3!}{(r-1)(2r-1)!} \Big\{  \kappa_{2r}(F) - \kappa_{2r}(F_\infty) \Big\}$.
\item $\frac{r}{r-1} \Big\{ \frac{\kappa_{2r}(F)}{(2r-1)!} -1 \Big\} \le \Big\{ \frac{\kappa_{2r+2(F)}}{(2r+1)!} -1 \Big\}$.
\end{enumerate}
Furthermore, assume that $\kappa_2(F)=1$,and $\kappa_4(F) \ge 6$, then for all $r\ge 1$
\begin{equation}\label{eq:even-cumulants-2}
\kappa_{2r}(F) \ge \kappa_{2r}(F_\infty)=(2r-1)!,
\end{equation}
and, if $(\ref{eq:even-cumulants-2})$ becomes equality for {\bf some} $r \ge 3$, then it becomes equality for {\bf all} $r \ge 3$.
\end{prop}

\begin{proof}
First note that using relation $(\ref{eq:1})$,
\begin{equation*}
\begin{split}
0 &\le \text{Var} \left( \Gamma_2(F) - F \right) \le \sum_{2 \le s \le r-1} \text{Var} \left( \Gamma_s(F) - \Gamma_{s-2}(F)\right)\\
& = \Big\{ \frac{\kappa_{2r}(F)}{(2r-1)!} - \frac{\kappa_{2r-2}(F)}{(2r-2)!} \Big\} -\Big\{  \frac{\kappa_4(F)}{3!}- \kappa_2(F)  \Big\}.\\
\end{split}
\end{equation*}
Hence $$ \frac{\kappa_{2r}(F)}{(2r-1)!} - \frac{\kappa_{2r-2}(F)}{(2r-2)!} \ge \frac{\kappa_4(F)}{3!}- \kappa_2(F), \quad r \ge 2.$$ Using a telescopic  argument yields that 
\begin{equation}\label{eq:telescope}
\frac{\kappa_{2r}(F)}{(2r-1)!}  - \kappa_2 (F) = \sum_{2 \le s \le r}    \Big\{ \frac{\kappa_{2s}(F)}{(2s-1)!} - \frac{\kappa_{2s-2}(F)}{(2s-2)!} \Big\} \ge (r-1)  \Big\{ \frac{\kappa_4(F)}{3!}- \kappa_2(F)\Big\}.
\end{equation}
 Next, we prove item $2$. We proceed with induction on $r \ge 2$. Let $r=2$, and we assume that $\kappa_2(F)=1$, then 
$$0 \le \Delta_{3,1}(F)= \frac{\kappa_6(F)}{5!} - 2 \frac{\kappa_4(F)}{3!} + 1 = \left( \frac{\kappa_6(F)}{5!} -1 \right) - 2 \left( \frac{\kappa_4(F)}{3!} -1 \right).$$ Similarly, using induction hypothesis, 

\begin{equation*}
\begin{split}
0 \le \Delta_{r+1,r-1}(F)& =\frac{\kappa_{2r+2}(F)}{(2r+1)!} -2 \frac{\kappa_{2r}(F)}{(2r-1)!} + \frac{\kappa_{2r-2}(F)}{(2r-3)!}\\
&= \left( \frac{\kappa_{2r+2}(F)}{(2r+1)!} -1 \right) -2 \left(  \frac{\kappa_{2r}(F)}{(2r-1)!}  -1 \right) + \left( \frac{\kappa_{2r-2}(F)}{(2r-3)!} -1 \right) \\
& \le \left( \frac{\kappa_{2r+2}(F)}{(2r+1)!} -1 \right) -2 \left(  \frac{\kappa_{2r}(F)}{(2r-1)!}  -1 \right) + \frac{r-2}{r-1} \left( \frac{\kappa_{2r}(F)}{(2r-1)!} -1 \right), 
\end{split}
\end{equation*}
which implies the claim. Item $1$ can be also shown in similar way. Moreover $(\ref{eq:even-cumulants-2})$ is a direct application of $(\ref{eq:even-cumulants-1})$. Finally, if 
$\kappa_{2r} (F)= \kappa_{2r}(F_\infty)=(2r-1)!$  for some $r \ge 3$, then estimate $(\ref{eq:telescope})$ implies that $\kappa_{2s} (F)= \kappa_{2s}(F_\infty)=(2s-1)!$ for all $2 \le s \le r$, and in particular $\Delta_{3,1}(F)=0$. Hence, Corollary \ref{cor:chain} yields that $\Delta_{r+1,r-1}(F)=0$, hence $\kappa_{2r+2}(F)= \kappa_{2r+2}(F_\infty)=(2r+1)!$, and so on.
\end{proof}

\begin{prop}\label{prop:moments-estimate}
Let $F$ be a generic element in the second Wiener chaos such that $\E(F^2)=\kappa_2(F)=1$, and $\E(F^4)\ge 9$ (or equivalently $\kappa_4(F) \ge 6$). Then 
\begin{equation}\label{eq:moments-estimate}
\mu_{2r}(F):= \E(F^{2r}) \ge \mu_{2r} (N_1 \times N_2) = \big( (2r-1)!! \big)^2 = \left( \frac{(2r)!}{r! 2^r} \right)^2.
\end{equation}
If $(\ref{eq:moments-estimate})$ is an equality for some $r \ge 3$, it holds as equality for all $r \ge 1$. In
such case we have also $\E(F^{2r+1})=\kappa_{2r+1}(F)=0$ for all $r \ge 1$, and necessarily $F \stackrel{\text{law}}{=} N_1 \times N_2$.
\end{prop}

\begin{proof}
First we recall that moments and cumulants are related by
\begin{align}\label{cumulant2moment}
  \mu_n (F)=\E(F^n)= \sum_{\pi \in \Pi_n} \prod_{A\in \pi} \kappa_{|A|}(F)
\end{align}
where the sum is over the set of partitions of $\Pi_n$ of the set $[n]:=\{1,\cdots,n\}$, and the product is over the partition components.
Note also that $\kappa_{n}(\alpha F)=\alpha ^n \kappa_{n}(F)$ for any scalar $\alpha$ and  $\kappa_n( F) = \kappa_n(F_+)+ \kappa_n(-F_-)$,  
since  $F=(F_+ - F_-)$ with independent $F_{\pm}$. Next we compare the even moments of $F$ with the even moments of $N_1N_2$  by using  $(\ref{eq:even-cumulants-2})$. 
Since by assumption $F$ and $N_1 \times N_2$ have the same 2nd moment and $\E(F^4) \ge \E( N_1^4)^2=9$, necessarily also $\kappa_4(F)\ge \kappa_4(N_1\times N_2)$
and from the cumulant inequalities $(\ref{eq:even-cumulants-1})$ it follows that
\begin{align}\label{cumulant:ineq:2}
 \kappa_{2n}(F ) \ge \kappa_{2n}( N_1\times  N_2), \quad \forall n \ge 3
\end{align}
and if this inequality is an equality for some $n\ge 3$, it holds as equality $\forall n \ge 3$ as well. Note that
\begin{align}  \label{cumulant2moment:split}
   \mu_{2n}=\E(F^{2n})= \sum_{\pi \in \Pi_{2n}^' } \prod_{A\in \pi} \kappa_{|A|}(F)+ \sum_{\rho \in \Pi_{2n}^{''} } \prod_{B\in \rho} \kappa_{|B|}(F)
\end{align}
where $\Pi_{2n}^'$ are the partitions of $2n$ containing only components of even size, and $\Pi_{2n}^{''} = \Pi_{2n}\setminus \Pi_{2n}^'$ is its complement,
whose partition elements   contain a non-zero even number of components with odd size.
By \eqref{cumulant:ineq:2}, it is clear that
\begin{align*}
  \sum_{\pi \in \Pi_{2n}^' } \prod_{A\in \pi}  \kappa_{|A|}(F)  \ge    \sum_{\pi \in \Pi_{2n}^' } \prod_{A\in \pi}  \kappa_{|A|}(N_1 N_2) 
\end{align*}
when all sets $A$ have even size. In order to show that $\mu_{2n}(F) \ge \mu_{2n} ( N_1 N_2)$, it is enough to show that under the assumptions all the
odd cumulants of $F$ have the same sign or vanish, namely
\begin{align*}
 \kappa_{2n+1}(F) \kappa_{2m+1}(F) = \bigl( \kappa_{2n+1}( F_+) - \kappa_{2n+1}(F_-) \bigr) \bigl( \kappa_{2m+1}( F_+) - \kappa_{2m+1}(F_-) \bigr) 
 \ge 0 \quad \forall n > m,
\end{align*}
implying that the second sum in   \eqref{cumulant2moment:split}   
is  always non-negative.



The condition $\E(F^4)\ge 9$, which together with $\E(F^2)=1$, $\E(F)=0$ is equivalent to $\kappa_4(F)\ge 6$, implies that 
\begin{align*}
   \lambda_{-1}^2 \le  \frac 1 2 \le \lambda_{1}^2 \quad \mbox{ or } \quad \lambda_{1}^2 \le  \frac 1 2 \le \lambda_{-1}^2 . 
\end{align*}
By contradiction, assume that  $\lambda_{z}^2 < 1/2$  strictly $\forall z$, which implies
\begin{align*}
 \frac 1 2=\frac{\kappa_4( N_1 N_2)}{ 12}\le  \sum_{z\in \Z} \lambda_{z}^4 < \frac  1 2  \sum_{z\in \Z} \lambda_{z}^2 = \frac {\E(F^2) } 2 = \frac 1 2
\end{align*}
with strict inequality, which is a contradiction. Therefore, $\exists z\in \Z$ with $\lambda_{z}^2 \ge 1/2$. Assume without loss of generality that 
\begin{align*}
    \frac 1 2\le  \lambda_1^2 \le   \sum_{\ell \in \N} \lambda_{\ell}^2  =       \E( F_+^2) =  1- \E( F_{-}^2).
\end{align*}
Then, $\forall \ell \ge 1$
\begin{align*}
 \lambda_{-\ell}^2 \le \sum_{\ell \in \N} \lambda_{-\ell}^2= \E( F_{-}^2) \le \frac 1 2 \; .
\end{align*}
Now it follows $\kappa_n(F_+) \ge \kappa_n(F_-)$ $\forall n \ge 2$, since 
\begin{align*}
 \sum_{\ell \in \N} \lambda_{\ell}^n \ge \lambda_{1}^n \ge  2^{-n/2} \ge \biggl( \sum_{\ell\in \N} \lambda_{-\ell}^2 \biggr)^{n/2} \ge \sum_{\ell\in \N} |\lambda_{-\ell}|^n
\end{align*}
where the last inequality is referred as Jensen inequality for  sequences, which is strict unless the series has at most one nonzero term \cite{hardy}. 

\end{proof}

\begin{cor}\label{cor:4-2r-moments}
For a random element $F$ in the second Wiener chaos with $\E(F^2)=1$, $\E(F^4) \ge 9$, and $\E(F^{2r}) \le \big( (2r-1)!!\big)^2$ for some $r \ge 3$ necessarily we have $F \stackrel{\text{law}}{=} N_1 \times N_2$.
\end{cor}

\begin{rem}{ \rm
It worth to separately point it out that the random variable $F_\infty \sim N_1 \times N_2$ minimizes all the even moments/cumulants among the class of random elements in the second Wiener chaos having the moment properties $\E(F^2)=1$, and $\E(F^4)\ge 9$, see also Proposition \ref{prop:sym-case}.
}
\end{rem}

\begin{rem}{ \rm
The assumption $\E(F^4) \ge 9$ in Corollary \ref{cor:4-2r-moments} is essential and cannot be dropped.
For example, consider a random element $F$ in the second Wiener chaos with three non zero coefficients
$\lambda_1 = 0.7624,\lambda_2 = 0.5370, \lambda_{-1} = 0.3610$, i.e $$F=\frac {1}{\sqrt{2}} \bigl\{ 
\lambda_1 (N^2_1 -1) + \lambda_2 (N^2_2 -1) - \lambda_{-1} (N^2_{-1} -1) \bigr\}$$ where $N_1, N_2, N_{-1} \sim \mathscr{N}(0,1)$ are independent.  We found these $\lambda_i$ values by minimizing numerically the 4-th moment with 2nd and 6th moment constraints. For such random variable $F$  (up to  numerical precision) we get we get $\E(F^2) = 1, \E(F^6) = (5!!)^2 = 225$, and obviously $F$ is not distributed as $N_1 \times N_2$. This is because of $\E(F^4) = 8.2567 < 9$.
}
\end{rem}

\begin{prop}\label{prop:difference-moments-estimate}
Under the assumptions of Proposition \ref{prop:moments-estimate},  for $2 \le m\le n  \in \N$, we have 
$$\mu_{2n}(F) - \mu_{2n}(N_1\times N_2) \ge {2n-2m \choose 2}  \Big( \mu_{2m}(F) - \mu_{2m}(N_1 \times N_2) \Big).$$
\end{prop}

\begin{proof} By the cumulants-to-moments formula
\begin{align*}
  \mu_{2n}(F) - \mu_{2n}(N_1 N_2)= \sum_{\pi \in \Pi_{2n} } \biggl\{  \prod_{A\in \pi} \kappa_{|A|}(F) - \prod_{A\in \pi} \kappa_{|A|}(N_1 N_2) \biggr\}.
 \end{align*}
Now for each partition $\pi \in \Pi_{2n}$,
\begin{align*}
 \prod_{A\in \pi} \kappa_{|A|}(F) \ge \prod_{A\in \pi} \kappa_{|A|}(N_1 N_2).
 \end{align*}
Indeed if the partition $\pi$ contains any part $A$ with odd size, then the right side is  zero,
and the left side is non-negative since even cumulants are non-negative, there must be an even number of odd parts in
the partition and under the assumptions  all odd cumulants have the same sign. Otherwise the partition $\pi$ contains only parts of even size, but then we have shown that under the assumptions
\begin{align*}
 \kappa_{2\ell}(F) \ge \kappa_{2\ell}(N_1 N_2) \ge 0,  \quad  \forall \ell \in \N,
\end{align*}
and the inequality is preserved when we take product over the partition. If  $\pi$ is  partition of $2m$, let's say $\pi= \{ A_1 , A_2, \dots , A_r\}$ with $A_i \cap A_j = \emptyset $ for $i\ne j$
and $A_1 \cup A_2 \cup \dots \cup A_r = \{ 1,2,\dots, 2m \}$, then we can add $(n-m)$ pairs to obtain
\begin{align*}
 \pi^' =  \bigl\{  A_1 ,A_2, \dots , A_r, \{ 2m+1, 2(m+1)\} ,\dots , \{ 2n-1, 2n \} \bigr\}
\end{align*} which
is a partition of $\{  1,2,\dots, 2n\}$,
and 
\begin{align*}
 \prod_{A^'\in \pi^'} \kappa_{A^'}( F) = \kappa_2( F)^{n-m}\prod_{A\in \pi} \kappa_A( F)  = \prod_{A\in \pi} \kappa_A( F) .
\end{align*}
Since we could choose those pairs differently,  
for every partition of $2m$ there  are at least  $\binom{2 n-2m }{2}$ partitions of $2n$ which contribute equally to the sum, and  we get
\begin{align*}
 \sum_{\pi^' \in \Pi_{2n} }   \underbrace{ \biggl(  \prod_{A^'\in \pi^' } \kappa_{|A^'|}(F) - \prod_{A^'\in \pi^'} \kappa_{|A^'|}(N_1 N_2)     
 \bigr ) }_{\ge 0}  \ge  \binom{2 n-2m }{2} \sum_{\pi \in \Pi_{2m} }   \underbrace{ \biggl(
 \prod_{A\in \pi} \kappa_{|A|}(F) - \prod_{A\in \pi} \kappa_{|A|}(N_1 N_2)       \biggr) }_{\ge 0}  
\end{align*}
\end{proof}
%
%
In what follows the notation $\mathscr{H}^{\text{symm}}_2$ stands for the collection of random variables in the second Wiener chaos with symmetric distributions. 
\begin{prop}\label{prop:sym-case}
Let $F \in \mathscr{H}^{\text{symm}}_2$ such that $\E(F^2) \le 1$. Then 
\begin{enumerate}
\item for $ r \in \N$, we have $\kappa_{2r}(F) \le \kappa_{2r}(N_1 \times N_2)$.
\item for $r \in \N$, we have $\mu_{2r} (F) \le \mu_{2r} (N_1 \times N_2)$.
\item if one of these cumulant or moment inequalities at items $1$ or $2$ is an equality for some $r \ge 2$, then $F \stackrel{\text{law}}{=} N_1 \times N_2$.
\end{enumerate}
\end{prop}

\begin{proof}   
Since  $\lambda_{\ell}=- \lambda_{-\ell}$  for $F\in\mathscr{H}^{\text{symm}}_2, \ell\in \N$, and $\E(F^2)\le 1$, for $r \in \N$ by using Jensen inequality
\begin{equation*}
\begin{split}
\frac{ 2^{1-r} }{(2r-1)!}  \kappa_{2r}( F)  &= \sum_{\ell\in \Z } \lambda_{\ell}^{2r} =  2 \sum_{\ell  \in \N} \lambda_{\ell}^{2r}\le  2  \biggl( \sum_{ \ell \in \N } \lambda_{\ell}^2 \biggr)^r = 2^{1-r} \E \bigl( F^2\bigr)^r  \\
& \le \frac{ 2^{1-r} }{(2r-1)!}  \kappa_{2r}( N_1\times N_2)
\end{split}
\end{equation*}
with equality if and only if  $\lambda_{\ell}=0$ $\forall \ell \not\in \{ -1,1\}$. This proves item $1$, with equality if and only if $F \stackrel{\text{law}}{=} N_1 \times N_2$ 
and hence the half of the item $3$. Since $F$ is symmetric, the odd cumulants are zero, and  by the cumulant to moments formula we obtain 
\begin{align}\label{cum2moments:ineq}
 \mu_{2r}(F) = \sum_{\pi \in  \Pi_{2r} } \prod_{A\in \pi}  \kappa_{|A|}( F ) \le  \sum_{\pi \in  \Pi_{2r} } \prod_{A\in \pi}  \kappa_{|A|}( N_1\times N_2 )=\mu_{2r}(N_1 \times N_2),
\end{align}
which is an equality if and only if $\kappa_{2s}( F )=  \kappa_{2s}( N_1 \times N_2 )$ $\forall s\le r$, meaning that $F \stackrel{\text{law}}{=} N_1 \times N_2$. Hence item $2$ is shown together with the remaining half part of item $3$. 
\end{proof}

\subsection{Case $\E(F^4)<9$}
In this section, we aim to cover the case when $F_n$ is a sequence of random elements in the second Wiener chaos such that  $\liminf_{n} \E( F_n^4)\le \E\bigl((N_1\times N_2)^4\bigr)=9$. 
For example, imagine the case when $\mu_4 (F_n) \to 9$ from below as $n \to \infty$. We start with the following useful  
observation on the geometry of $\ell^p$ spaces. Let $p >0$. For a sequence $x$ we denote $\Vert x \Vert_p:= \left( \sum_{i \ge 1} \vert x^p_i \vert \right)^{\frac{1}{p}}$,
$\Vert x \Vert_{\infty}:= \sup_i \bigl\{ \vert x_i \vert \bigr\}$.
\begin{lem}\label{lem:analytic}
Let $\epsilon < \frac{1}{6}$. Assume  $x=(x_1,x_2,\cdots) \in \R^{\N}$ such that 
$\Vert x \Vert_{1}=1, \Vert x \Vert_\infty < \frac{1}{2}$, and $\Vert x \Vert^2_2 > \frac{1}{2} - \epsilon$. Then, there are exactly two indices $ k \neq l $ such 
that $\frac{1}{2} - \vert x_k  \vert< \epsilon, \frac{1}{2} - \vert x_l \vert< \epsilon$, and $\sum\limits_{i\neq k,l} \vert x_i  \vert< 2\epsilon$.
\end{lem}
\begin{proof}
Assume $\epsilon >0$ is sufficiently small. We will make it clear at the end.  Without loss of generality assume that $x_i\ge 0$ $\forall i$, (otherwise consider the sequence $\vert x_i\vert$).
Denote 
$$A:= \sum_{i \ge 1}  \bigl(\frac{1}{2} - x_i \bigr)x_i.$$ Then $ 0 < A= \frac{1}{2} \Vert x \Vert_1 - \Vert x \Vert^2_2 = 
\frac{1}{2} - \Vert x \Vert^2_2 < \epsilon$. Set $I =\{ i \ge 1 \, : \, x_i > \frac{1}{2} - \epsilon \}$. Then $I \neq \emptyset$, 
otherwise $x_i \le \frac 1 2 - \epsilon$ for all $ i \ge 1$, and therefore $\Vert x \Vert^2_2 \le \bigl( \frac 1 2 -\epsilon \bigr)\sum_{i \ge 1} x_i =
\frac 1 2 -\epsilon$ which is a contradiction. Next, we show that $\# I \ge 2$. By contrary assume that $\# I =1$, and $j \in I$. Then
$$\Vert x \Vert^2_2 = \sum_{i \ge 1} x^2_i = x^2_j + \sum_{i \neq j} x^2_i < x^2_j + \epsilon ( 1- x_j) < \frac{1}{4} + \epsilon ( \epsilon + \frac{1}{2}) < \frac{1}{2} - \epsilon$$ for every $\epsilon < \frac{1}{4}$, which is again a contradiction. Obviously $\# I < \infty$, and now we are going to show that in fact $\#I =2$. To this end, note that
\begin{equation*}
\begin{split}
\epsilon > A \ge \sum_{i \in I } \bigl(\frac{1}{2} - x_i  \bigr) x_i & > (\frac{1}{2} - \epsilon) \sum_{i \in I} (\frac{1}{2} - x_i)\\
& = (\frac{1}{2} - \epsilon)  \big\{ \frac{\#I}{2} - \sum_{i \in I} x_i \big\} \\
& \ge (\frac{1}{2} - \epsilon) \big\{ \frac{\#I}{2} - 1\big\}.
\end{split}
\end{equation*}
Hence $\#I =2$, otherwise for $\epsilon < \frac{1}{6}$ the above chain of inequalities do not takes place. 
\end{proof}
Take an element $F$ in the second Wiener chaos. As it indicates in the proof of Proposition \ref{prop:moments-estimate} the key point to control the signs of 
the products of the odd cumulants of $F$ was to realize at least one coefficient $\lambda_i$ in the representation of $F$ such that $\lambda^2_i \ge \frac{1}{2}$. The next corollary studies the situation that all $\lambda^2_i \le \frac{1}{2}$.



\begin{cor}\label{prop:4<9}
Let $\epsilon < \frac{1}{72}$, and  $F=\sum_{i \in \Z} \lambda_i (N^2_i -1)/ \sqrt 2$ be a random variable in the second Wiener chaos 
such that $\E(F^2)=1$,  and 
$9 \ge \E(F^4) > 9 - \epsilon$ (or equivalently $6 \ge \kappa_4 (F) > 6 - \epsilon$). If  $ \lambda_i^2 \le  1/2$ for all $i \ge 1$,
there exist exactly two indices $k \neq l$ such that 
\begin{itemize}
\item[(i)] $0 \le  1/2 - \lambda^2_k < \epsilon$, and also $ 0 \le 1/2 - \lambda^2_l  < \epsilon$.
\item[(ii)] $ \sum_{i \neq k,l}\lambda^2_i < 2\epsilon $. 
\end{itemize}
\end{cor}
\begin{proof}
This is a direct application of Lemma \ref{lem:analytic} with $x_i = \lambda^2_i$
\end{proof}

\begin{rem}{ \rm One has to note that under the assumptions of Corollary \ref{prop:4<9} even for very tiny $\epsilon>0$ the laws 
of $F$ and $N_1 \times N_2$ might be very different. For example, consider the simple random variable $F=\frac{1}{ 2}(N^2_1 -1 ) +
\frac{1}{ 2}(N^2_2 -1 )$ where $N_1, N_2 \sim \mathscr{N}(0,1)$ are independent. We get $\E(F^2)=1$, and $\E(F^4) =9$ matching the fourth moment of $N_1\times N_2$,
however  $F$  has centered chi squared distribution with two degrees of freedom.  
Indeed this observation highlights the role of a higher even moment matching.}
\end{rem}  

The next lemma is a well know fact in the Wiener analysis for all chaoses, see for example \cite[Corollary 2.8.14]{n-p-book}. However, to be self-contained, we provide a simple proof of the fact in the case of the second Wiener chaos.  We will use it in Section \ref{sec:main-results}.

\begin{lem}(hypercontractivity) \label{classical:hypercontractivity} The cumulants and moments of a r.v. $F\in \mathscr{H}_2$ satisfy
\begin{align*}
\vert \kappa_{n}(F) \vert\le  2^{n/2-1} (n-1)! \kappa_{2}(F)^{n/2}, \quad   \vert \E( F^{n} ) \vert \le C_{n} \E( F^2)^{n/2}
\end{align*}
with constants
\begin{align*}
 C_n = 2^{n/2} \sum_{\pi \in \Pi_n} 2^{-|\pi| } \prod_{A\in \pi} ( |A|-1)! =2^{-n/2}\sum_{k=0}^n \binom{n}{k}(-1)^{n-k} \frac{(2k)!}{k! 2^k}
\end{align*}
with equalities if and only if  $F\stackrel{law}{=} \pm ( N^2 -1)/\sqrt{2}$,
where $N \sim \mathscr{N}(0,1)$ is standard Gaussian.
\end{lem}
\begin{proof} For $F\in \mathscr{H}_2$  with representation  \eqref{spectral:representation:classical}, by applying Jensen inequality
to \eqref{classical:spectral:representation} we obtain
\begin{align*}
 \vert\kappa_{n}(F)\vert= 2^{n/2-1}(n-1)!\bigg\vert\sum_{\ell\in \Z} \lambda_{\ell}^{n}\bigg\vert \le 2^{n/2-1}(n-1)!
 \biggl(  \sum_{\ell\in \Z} \lambda_{\ell}^2\biggr)^{n/2} =  2^{n/2-1}(n-1)!     \kappa_2(F)^{n/2}
 \end{align*}
 with equality if and only if the series have at most one nonzero term. Then the claim follows by the cumulant to moment relation \label{cumulant2moment}.
\end{proof}

\section{Free probability and the ''tetilla law''}\label{sec:fp}

We introduce some basic notions of  non-commutative probability theory, 
following very  closely \cite[Ch.8]{n-book},\cite{n-p-2w}.
A free probability space is a pair $({\mathcal A},\varphi)$, where ${\mathcal A}$ is a Von-Neumann algebra (that is, an algebra of bounded operators on a complex
separable Hilbert space, closed under the adjoint and convergence in weak operator topology) and 
a trace $\varphi : {\mathcal A} \to \R$,  that is 
a weakly continuous linear operator satisfying $\varphi( {\bf 1}) = 1$, which is  tracial (meaning that $\varphi( XY)=\varphi(YX)$ $\forall X,Y\in {\mathcal A})$,
positive and faithful (meaning that $\varphi( X X^*) \ge 0$  $\forall X\in {\mathcal A}$, with equality
if and only if $X=0$).
Elements of the algebra  ${\mathcal A}$ are called {\it non-commutative random variables}.

We say that the unital subalgebras ${\mathcal A}_1, \dots, {\mathcal A}_n$ of ${\mathcal  A}$ are {\it freely independent}
when  the following property holds:  $\forall m$, $\forall X_1,\dots, X_m$ such that $\varphi(X_i)=0$ and $X_i \in {\mathcal A_j}$ for some $1\le j \le m$,
and, $\forall i=1,\dots, m-1$,  consecutive $X_i,X_{i+1}$ do not belong to the same ${\mathcal A}_j$ subalgebra, then
$\varphi( X_1 X_2 \dots X_n )=0$. We say that the non-commutative random variables $X_1,\dots, X_n$ are  freely independent
if the unital subalgebras they generate are freely independent.
If $X,Y$ are free, we have $\varphi(X^n Y^n)= \varphi(X^n) \varphi(Y^n)$ as in the classical case,
however $\varphi\bigl( ( X Y)^2\bigr)=\varphi(Y)^2 \varphi(X^2) + \varphi(Y^2) \varphi(X)^2-\varphi(Y)^2 \varphi(X)^2$. 
We remark that classical   probability is included in free probability theory as a special case, when we consider
\begin{align*}
  {\mathcal A}= \bigcap_{p< \infty} L^p( \Omega,{\mathcal F},\P) ,\quad  \varphi( X ) = \E(X).
\end{align*}
A partition $\rho$ of $\{ 1, \dots,n \}$ is said to be {\it non-crossing } if there are integers  $1\le p_1< q_1 < p_2 < q_2 \le n$
such that $p_1, p_2$  are in the $\rho$-partition block $B$, and $q_1, q_2$ are in the $\rho$-partition block $B^'$,  then
necessarily  $B=B^'$.

Moments $\mu_n(F)$ and {\it free cumulants  }   $\widehat\kappa_{\ell}(F)$ of  
a non-commutative random variable $F$ are  defined by the relations
\begin{align} \label{free:cumulant2moment}
  \mu_n(F):=\varphi(F^n)= \sum_{\rho \in {NC}_n} \prod_{A\in \rho}\widehat\kappa_{|A|}(F).
\end{align}
where  the sum is over  
the   non-crossing  partitions of $\{ 1,\dots,n\}$. 
M\"obius inversion  formula is given by
\begin{align*} \widehat 
 \kappa_n(F)= \sum_{\rho \in {NC}_n} (-1)^{|\rho|-1} {\mathcal C}_{|\rho|-1} \prod_{A\in \rho}\mu_{|A|}(F).
\end{align*}
where ${\mathcal C}_n = \binom{2n}{n}/(n+1)= \# \{  \mbox{ non-crossing partitions of $n$ } \}$  are  the Catalan numbers \cite{kre72}.

\subsection{Semi-circular law}
Following \cite{nica-speicher-book},
we say that a probability distribution $Q$ on $(\R,{\mathcal B}(\R))$ is the law of  the non-commutative random variable $F$
if
\begin{align*} 
   \varphi( F^n ) = \int_{\R} x^n  Q(dx), \quad  \forall n \in \N.
\end{align*}

The {\it semicircle law} of parameter $t>0$ has density with respect to the Lebesgue measure given by
\begin{align*}
 q_t(x)=\frac 1 {2\pi t} \sqrt{ 4t -x^2}  {\bf 1}( -2 \sqrt t < x < 2\sqrt t ).
\end{align*}
Since the semicircle law is symmetric, the odd moments vanish, and for the even moments we have 
\begin{align*}
 \int_{-2 \sqrt  t} ^{ 2 \sqrt t}  x^{2n}  q_t(x) dx = {\mathcal C}_n t^n.
\end{align*}
In particular a classical or non-commutative $t$-semicircular random variable $S(t)$ has  $\mu_2(S(t))=t$ and $\mu_4(S(t))=2 t^2$.
The  free cumulants to moment relation \eqref{free:cumulant2moment} implies that $\widehat\kappa_2( S(t))= t$ and $\widehat\kappa_n(S(t))=0$, for all $n \ne 2$.
In non-commutative probability the semi-circular law plays the same role as  
the Gaussian law in  classical probability.

\subsection{Tetilla law}
The {\it tetilla law }  is the distribution of the  non-commutative random variable 
\begin{align}\label{tetilla}
F_{\infty}:=\frac{ S_1^2-S_2^2 }{ \sqrt 2} \stackrel{law}{=} \frac{  S_1 S_2 +S_2 S_1 }{  \sqrt 2},
\end{align}
where $S_1$ and $S_2$ are freely independent semicircular random variables with unit variance. It was studied first in \cite{nica-speicher-98}, and it takes its name from the resemblance
of the density function with the anatomical profile 
\cite{d-n}. It can be shown that $\widehat\kappa_{n}(S(t)^2)=t^n$ \cite[Proposition 12.13]{nica-speicher-book}.
By symmetry $\varphi(F_{\infty}^{2n+1})=\widehat\kappa_{2n+1}(F_{\infty})=0$, while 
the free cumulants and moments  of the tetilla law are  
\begin{align*} \quad \widehat\kappa_{2n}( F_{\infty}) = 2^{1-n}, \quad
 \varphi\bigl( F_{\infty}^{2n} )= \frac 1 {2^n n} \sum_{k=1}^n 2^k \binom{n}{k} \binom{2n}{k-1}, \qquad \text{see \cite{d-n}.}
\end{align*}

\subsection{Free Brownian motion} A free Brownian motion on the non-commutative  $(\mathcal A, \varphi)$  consists in a filtration 
$( {\mathcal A}_t )_{t\ge 0}$, which is a  sequence of unital sub-algebra of ${\mathcal A}$ with ${\mathcal A}_u \subset {\mathcal A}_t$  for $0\le u < t$,
and a collection of self-adjoint operators $(S(t): t \ge 0)$  such that 
\begin{enumerate}
 \item[(1)]  $S(t) \in {\mathcal A}_t$ $\forall t$,
 \item[(2)]  each $S(t)$ has the semicircular law with parameter $t$,
 \item[(3)] for every $0\le u \le t$  the increment $\bigl(S(t)-S(u)\bigr)$ is freely independent from ${\mathcal A}_u$ and it has the semicircular law with variance parameter $(t-u)$.
\end{enumerate}
The free Brownian motion can thought as a matrix-valued Brownian motion in infinite dimension.

\subsection{Characterization of the tetilla law: case $\varphi(F^4)\ge 5/2$}
Throughout this section, the random element $F_\infty$ distributed as a normalized tetilla law given as $(\ref{tetilla})$. In analogy with the Wiener chaos with respect to classical Brownian motion,
 the $q$-th Wigner chaos with respect to free Brownian motion is constructed in \cite{b-s}  as follows:
for a simple function of the form
\begin{align*}
f(t_1,\dots,t_q) =  {\bf 1}( a_1< 	t_1<  b_1)  \times \dots \times  {\bf 1}( a_q< t_q< b_q)
\end{align*}
with $0\le a_1 < b_1 \le a_2  < b_2 \le \dots \le a_q < b_q$
define
\begin{align*}
 I_q^S( f)= (S_{b_1}-S_{a_1} ) \bigl (S(b_2)-S(a_2) \bigr ) \dots  \bigl (S(b_q)-S(a_q) \bigr ).  
\end{align*}
Let $f \in L^2(\R^q_+)$, and define the {\it adjoint} function $f^{\star}(t_1,t_2,\cdots,t_q):= f(t_q,\cdots,t_2,t_1)$. In general, object $I^S_q(f)$ for $f\in L^2(\R_+^q )$ can be defined by a density argument, using linearity and  the isometry
\begin{align*}
 \langle I_q^S(f), I_q^S(g)  \rangle_{L^2( {\mathcal A},\varphi ) }:=
 \varphi(  I_q^S(f)^* I_q^S(g) )  =  \varphi(  I_q^S(f^*) I_q^S(g) ) = \int_{\R_+^q} f(x) g(x) dx= \langle f, g \rangle_{L^2( \R_+^q) },
\end{align*}
which follows immediately for  simple functions $f,g$ vanishing on diagonals.
Let $S=(S(t):t\ge 0)$  be a free Brownian motion defined on a non-commutative probability space $({\mathcal A},\varphi)$. Similarly, every element $F= I_2^S(f)$  with $f\in L^2_{\mbox{symm}}(\R_+^2)$ in the second Wigner chaos $\mathscr{H}^S_2$ allows the following representation
\begin{align}\label{2nd:wigner:chaos:representation}
 I_2^S(f) = \sum_{z\in \Z}  \lambda_{z}  \bigl( S_{z} ^2 -1)
\end{align}
where $S_{z}$ are freely independent centered semicircular non-commutative random variables
with unit variance and the series converges in $L^2({\mathcal A},\varphi)$ \cite[Proposition 2.3]{n-p-2w}.
The {\it free cumulants} of $F=I_2^S(f)$ are given by $\widehat\kappa_1(F)=0$ and
\begin{align}\label{free:cumulant:representation}
 \widehat\kappa_n (F)=\sum_{z\in \Z} \lambda_{z}^n,\quad n\ge 2.
\end{align}
 For the free cumulants, we have  if $F\in \mathscr{H}_2^S$, with $\varphi(F^2)=1$ and $(nm)\in 2\N$
\begin{align*}
 2^m \widehat\kappa_{2m}(F)+2^n\widehat\kappa_{2n}(F)-2^{(n+m+2)/2}\widehat\kappa_{n+m}(F)\ge 0
\end{align*}
with equality if $F=F_{\infty}$.
As before, by telescoping 
\begin{align*}
 2^{n-1} \widehat\kappa_{2n}(F)- \widehat\kappa_2(F)=\sum_{\ell=1}^{n-1}  \bigl( 2^{\ell}\widehat\kappa_{2\ell+2}(F)-2^{\ell-1}\widehat\kappa_{2\ell}(F) \bigr) \ge (n-1) \bigl( 2\widehat\kappa_4(F)- \widehat\kappa_2(F) \bigr)
\end{align*}
since
\begin{align*}
  2^n \widehat\kappa_{2n}(F) - 2^{n-1}\widehat\kappa_{2n-2}(F) \ge 2^{n-1}\widehat\kappa_{2n-2}(F) -2^{n-2} \widehat\kappa_{2n-4}(F) \ge \dots \ge 4\widehat\kappa_4(F)-2 \widehat\kappa_2(F),
\end{align*}
and if one of these inequalities is an equality, and $\widehat{\kappa}_{2m+1}(F)=0$ for some $m \ge 1$,  then $F=F_{\infty}$ in distribution. We summarize these facts in the following lemma.

\begin{lem} \label{free:cumulant:lemma}  
Let $F\in {\mathcal H}_2^S$, with $\varphi(F^2)=1$ , $\widehat\kappa_4(F) \ge \widehat\kappa_4(F_{\infty} )=1/2$, $\widehat\kappa_{2m+1}(F)=0$, and $\widehat\kappa_{2n}(F)\le \widehat\kappa_{2n}(F_{\infty})$ for some $m\ge 1, \, n \ge 3$. Then, these inequalities are equalities and $F\stackrel{law}{=} F_{\infty}$.
\end{lem}

Next, we shall derive  the corresponding characterizations by using  moments  instead of free cumulants.
The following result extends \cite[Theorem 1.1.]{d-n}.
\begin{thm}\label{tetilla:moment:characterization} Let $F\in {\mathcal H}_2^S$
a non-commutative random variable in the second Wigner chaos such that $\varphi(F^2)=\varphi(F_{\infty}^{2} )=1$ and $\varphi(F^4)\ge \varphi(F_{\infty}^4)=5/2$
where $F_{\infty}=(S_1^2 - S_2^2)/\sqrt{2}$ has the tetilla law.
 Then 
 \begin{align*}
   \varphi(F^{2n}) \ge \varphi(F^{2n}_{\infty} ),\quad n \in \N,
 \end{align*}
and if this inequality is an equality for some $n\ge 3$, then $F=F_{\infty}$ in distribution.
\end{thm}
\begin{proof}
We  follow the steps of the proof for commutative random variables, just note that by \eqref{free:cumulant2moment} 
\begin{align*}
   \mu_{2n}(F)=\varphi(F^{2n})= \sum_{\pi \in {NC}_{2n}^' } \prod_{A\in \pi} \widehat\kappa_{|A|}(F)+ \sum_{\rho \in {NC}_{2n}^{''} } \prod_{B\in \rho} \widehat\kappa_{|B|}(F)
\end{align*}
where ${NC}_{2n}^'$ are the non-crossing partitions of $2n$ containing only components of even size, and ${NC}_{2n}^{''} = {NC}_{2n}\setminus {NC}_{2n}^'$ is its complement,
where the non-crossing partition elements   contain an even number of components with odd size.
As in the classical case, the problem is to deal with the free cumulants of odd order.
Note  that  $F=F_+ - F_-$ with free $F_{\pm}$, and
$\widehat\kappa_n( F) = \widehat\kappa_n(F_+)-\widehat\kappa_n(F_-)$, where $\widehat\kappa_{n}(\alpha F)=\alpha ^n \widehat\kappa_{n}(F)$.
As in the classical case, the assumptions $\varphi(F^2)=\varphi(F_{\infty}^{2} )=1$ and $\varphi(F^4) \ge \varphi(F^4_\infty)$ imply that all odd free cumulants have the same sign, i.e.
\begin{align*}
  \widehat\kappa_{2m+1}(F) \widehat\kappa_{2n+1}(F) \ge 0, \quad \forall \, n,m.
\end{align*}
Therefore, all the terms in the sum are non-negative and minorized by the corresponding  products of  $F_{\infty}$-free cumulants, and when
one of these even moment inequalities is an equality Lemma \ref{free:cumulant:lemma} applies. 
\end{proof}

\begin{prop}\label{prop:difference-moments-estimate:free}
Under the conditions of Theorem \ref{tetilla:moment:characterization},  for $2 \le m\le n  \in \N$, we have 
$$\varphi(F^{2n}) - \varphi(F^{2n}_{\infty} ) \ge {\mathcal C}_{n-m} 
\Big( \varphi(F^{2m}) - \varphi(F^{2m}_{\infty} ) \Big),$$
where ${\mathcal C}_{k}$ denotes the $k$-th Catalan number.
\end{prop}

\begin{proof} By using \eqref{free:cumulant2moment} 
\begin{align*}
 \varphi(F^{2n}) - \varphi(F^{2n}_{\infty} )= \sum_{\rho \in NC(2n) } \biggl\{  \prod_{A\in \rho} \widehat\kappa_{|A|}(F) - \prod_{A\in \rho} \widehat\kappa_{|A|}(F_{\infty} ) \biggr\}.
 \end{align*}
Now for each non-crossing partition $\rho \in {NC}_{2n}$,
\begin{align*}
 \prod_{A\in \rho} \widehat\kappa_{|A|}(F) \ge \prod_{A\in \rho} \widehat\kappa_{|A|}(F_{\infty} ). 
 \end{align*}
Indeed if the non-crossing partition $\rho$ contains any part $A$ with odd size, then the right side is  zero,
and the left side is non-negative since even free  cumulants are non-negative, there must be an even number of odd parts in
the partition and under the assumptions  all odd free cumulants have the same sign. 
Otherwise the non-crossing partition $\rho$ contains only parts of even size, but then we have shown that under the assumptions
\begin{align*}
 \widehat\kappa_{2\ell}(F) \ge \widehat\kappa_{2\ell}(N_1 N_2) \ge 0,  \quad  \forall \, \ell \in \N,
\end{align*}
and the inequality is preserved when we take product over the parts.
If  $\rho$ is a non-crossing partition of $2m$, let's say $\rho= \{ A_1 , A_2, \dots , A_r\}$ with $A_i \cap A_j = \emptyset $ for $i\ne j$
and $A_1 \cup A_2 \cup \dots \cup A_r = \{ 1,2,\dots, 2m \}$
then we can add $(n-m)$ pairs of consecutive elements to obtain
\begin{align*}
 \rho' =  \bigl\{  A_1 ,A_2, \dots , A_r, \{ 2m+1, 2(m+1)\} ,\dots , \{ 2n-1, 2n \} \bigr\}
\end{align*} which
is a non-crossing  partition of $\{  1,2,\dots, 2n\}$,
and 
\begin{align*}
 \prod_{A'\in \rho'} \widehat\kappa_{A'}( F) = \widehat\kappa_2( F)^{n-m}\prod_{A\in \rho} \widehat\kappa_A( F)  = \prod_{A\in \rho} \widehat\kappa_A( F) .
\end{align*}
Since  ${\mathcal C}_{n-m}$ is also the number of non-crossing pairings of $\{ 1,2,\dot, 2(n-m)\}$,
for every non-crossing  partition of $2m$ there  are at least  ${\mathcal C}_{n-m} $ non-crossing partitions of $2n$ which contribute equally  to the sum and we get
\begin{align*}
 \sum_{\rho' \in NC_{2n} }   \underbrace{ \biggl(  \prod_{A'\in \rho' } \widehat\kappa_{|A'|}(F) - \prod_{A'\in \rho'} \widehat\kappa_{|A'|}(F_{\infty} )     
 \biggr ) }_{\ge 0}  \ge  {\mathcal C}_{n-m} \sum_{\rho  \in NC_{2m} }   \underbrace{ \biggl(
 \prod_{A\in \rho } \widehat\kappa_{|A|}(F) - \prod_{A\in \rho } \widehat\kappa_{|A|}(F_{\infty} )       \biggr) }_{\ge 0}  
\end{align*}
\end{proof}
%
%

In what follows the notation $\mathscr{H}^{S,\text{symm}}_2$ stands for the collection of non-commutative random variables in the second Wigner chaos with symmetric distributions. 
\begin{prop}\label{prop:sym-case-non}
Let $F \in \mathscr{H}^{S,\text{symm}}_2$ such that $\varphi(F^2) \le 1$. Then 
\begin{enumerate}
\item for $ r \in \N$, we have $\widehat\kappa_{2r}(F) \le \widehat\kappa_{2r}(F_{\infty})$.
\item for $r \in \N$, we have $\varphi (F^{2r}) \le \varphi (F^{2r}_{\infty} )$.
\item if one of these free  cumulant or moment inequalities at items $1$ or $2$ is an equality for some $r \ge 2$, then $F \stackrel{\text{law}}{=} F_{\infty}$.
\end{enumerate}
\end{prop}
\begin{proof}   
Since  $\lambda_{z}=- \lambda_{-z}$  for $F\in\mathscr{H}^{\text{symm}}_2$, and $\varphi(F^2)\le 1$, for $r \in \N$ by using Jensen inequality
\begin{equation*}
\begin{split}
 \widehat\kappa_{2r}( F)  &= \sum_{z\in \Z } \lambda_{z}^{2r} =  2 \sum_{z  \in \N} \lambda_{z}^{2r}\le  2 
 \biggl( \sum_{ \ell \in \N } \lambda_{\ell}^2 \biggr)^r = 2^{1-r} \widehat\kappa_2( F)^r \le   \widehat\kappa_{2r}( F_{\infty} )
\end{split}
\end{equation*}
with equality if and only if  $\lambda_{z}=0$ $\forall z \not\in \{ -1,1\}$. This proves item $1$, with equality if and only if $F \stackrel{\text{law}}{=} 
F_{\infty}$ and hence the half of the item $3$. Since $F$ is symmetric, the odd free cumulants are zero, and  by the free cumulant to moments formula we obtain 
\begin{align}\label{cum2moments:ineq}
 \varphi(F^{2r}) = \sum_{\rho \in  NC_{2r} } \prod_{A\in \rho }  \widehat\kappa_{|A|}( F ) \le  \sum_{\rho \in  NC_{2r} } \prod_{A\in \rho} 
 \widehat\kappa_{|A|}( F_{\infty} )=\varphi(F^{2r}_{\infty} ),
\end{align}
which is an equality if and only if $\widehat\kappa_{2s}( F )=  \widehat\kappa_{2s}( F_{\infty})$ $\forall s\le r$, meaning that $F \stackrel{\text{law}}{=} F_{\infty}$.
Hence item $2$ is shown together with the remaining half part of item $3$. 
\end{proof}

\subsection{Case $\varphi(F^4) < 5/2$}
In this section, we aim to analysis the situation $\liminf_{n} \varphi\bigl( F_n^4\bigr)\le \varphi\bigl(F_{\infty}^4\bigr)=5/2$ for a sequence $\{F_n\}_{n \ge 1}$ of random elements in the second Wiener chaos.  For example, imagine the case when $\varphi(F^4_n) \to 5/2$ from below as $n \to \infty$. 
Take an element $F$ in the second Wigner chaos. 
As it indicates in the proof of Proposition \ref{tetilla:moment:characterization} 
the key point to control the signs of the products of the odd free cumulants of $F$ was to realize at least one coefficient 
$\lambda_i$ in the representation of $F$ such that $\lambda^2_i \ge \frac{1}{2}$. 
\begin{prop}\label{prop:4<9:free}
Let   $\epsilon < 1/72$,
and  $F=\sum_{i \ge 1} \lambda_i  (S_i^2-1)$ be a random variable in the second Wigner chaos 
such that $\varphi(F^2)=1, \vert \lambda_i \vert < \frac{1}{\sqrt{2}}$
for all $i \ge 1$, and $\varphi(F^4) > 5/2 - \epsilon$ (or equivalently $\widehat\kappa_4 (F) > 1/2- \epsilon$). Then there exist exactly two indices $k \neq l$ such that 
\begin{itemize}
\item[(i)] $\vert \lambda^2_k - \frac{1}{2} \vert < \epsilon$, and also $\vert \lambda^2_l - \frac{1}{2} \vert < \epsilon$.
\item[(ii)] $ \sum_{i \neq k,l} \lambda^2_i < 2\epsilon $ for all the other indices. 
\end{itemize}
\end{prop}
\begin{proof} As in Corollary \ref{prop:4<9}. 
\end{proof}

\begin{lem}(hypercontractivity) \label{free:hypercontractivity} The  free cumulants and moments of a non-commutative random variable  $F\in \mathscr{H}_2^S$
\begin{align*}
\vert \widehat\kappa_{n}(F) \vert\le \widehat\kappa_{2}(F)^{n/2}, \quad   \vert \varphi( F^{n} ) \vert \le {\mathcal C}_{n} \varphi( F^2)^{n/2}
\end{align*}
where ${\mathcal C}_{n}$ denotes the $n$-th Catalan number, with equalities if and only if  $F\stackrel{law}{=} \pm ( S_1^2 -1)$,
where $S_1$ has the circular law with unit variance or $F=0$.
\end{lem}
\begin{proof}As in Lemma \ref{classical:hypercontractivity}.
\end{proof}


\section{Convergence in Wasserstein-$2$ distance in  2nd Wiener/Wigner chaos }\label{sec:main-results}

The {\it Wasserstein$-2$ distance} between two probability distributions $Q_1,Q_2$ on $(\R,{\mathcal B}(\R) )$ is given by
\begin{align*}
  d_{W_2}( Q_1,Q_2):=\inf_{(X_1,X_2) } \biggl\{  \E\biggl(  (X_1-X_2)^2\biggr)^{1/2}   \biggr\}
\end{align*}
where the supremum is taken over the  random pairs  $(X_1,X_2)$ defined on the same classical probability spaces $(\Omega,{\mathcal F},\P)$
with marginal distributions $Q_1$ and $Q_2$. Relevant information about Wasserstein distances can be found, e.g.  in  \cite[Section 6]{villani-book}. It is shown in \cite[Thm 5.1]{b-v}  that
\begin{align*}
  d_{W_2}( Q_1,Q_2)=\inf_{(X_1,X_2) } \biggl\{  \varphi\bigl(   (X_1-X_2)^2\bigr)^{1/2}  \biggr\}  
\end{align*}
where the infimum is over the larger class of non-commutative r.v's $(X_1,X_2)$  defined on a common
non-commutative probability space $({\mathcal A},\varphi)$, with marginal laws $Q_1,Q_2$.

\subsection{Quantitative estimates in Wasserstein-$2$ distance}
\begin{prop}\label{prop:4>9}
Let $F$ be an element in the second Wiener (Wigner) chaos such that $\E(F^2)=1$ ($\varphi(F^2)=1$ respectively). Assume that $F_\infty \sim N_1 \times N_2$ (normalized tetilla law respectively) where $N_1, N_2 \sim \mathscr{N}(0,1)$ are independent. Then there exists a constant $C$ such that 
\begin{equation}\label{eq:4-6-moments}
d_{W_2} ( F, F_\infty ) \le_C \begin{cases} \sqrt{ \left( \mu_6(F) - 225 \right) - 55 \left( \mu_4 (F) - 9 \right)}, &\mbox{Wiener case}, \\
\sqrt{ \left( \varphi(F^6) - 8.25 \right) - 7 \left( \varphi(F^4) - 2.5 \right)}, &\mbox{Wigner case}. 
\end{cases}  
\end{equation}
If moreover assume that $\mu_4 (F) \ge 9$ ($\varphi(F^4) \ge 2.5$ respectively), then for every $r \ge 3$, we obtain $\mu_{2r}(F) \ge (2r-1)!!^2$, $(\varphi(F^{2r}) \ge \varphi(F^{2r}_\infty)$, where $F_\infty$ stands for normalized tetilla law), and in addition
\begin{equation}\label{eq:4-2r-moments}
d_{W_2} (F, F_\infty) \le_C \begin{cases}\sqrt{ \mu_{2r}(F) - (2r-1)!!^2}, & \mbox{Wiener case},\\
\sqrt{\varphi(F^{2r}) - \varphi(F^{2r}_\infty)}, & \mbox{Wigner case}.
\end{cases}
\end{equation}
\end{prop}
\begin{proof}
 Consider the polynomial $P(x) = x^6 - 55 x^4 + 331 x^2 - 61$. 
A straightforward computation yields that
$\E(P(F)) =5! \Delta_{3,1}(F) + 10 \kappa^2_3 (F) \ge 0$ in the light of $\Delta_{3,1}(F)=\text{Var}\left( \Gamma_2(F) - F\right) \ge 0$.
Now, relying on \cite[Theorem 2.4]{a-a-p-s} for some constant $C$ we obtain that 
\begin{equation*}
d_{W_2}(F,N_1 \times N_2) \le_C \sqrt{\E(P(F))} \le_C \sqrt{\left( \mu_6(F) - 225 \right) - 55 \left( \mu_4 (F) - 9 \right) }.
\end{equation*}
In particular, for $F$ in the second Wiener chaos with variance $1$, $$\left( \mu_6(F) - 225 \right) - 55 \left( \mu_4 (F) - 9 \right) \ge 0.$$ 
When $\E(F^4) \ge 9$, then Proposition \ref{prop:moments-estimate} tells us that the even moment $\mu_{2r}(F)\ge (2r-1)!!^2$ holds, 
and also estimate $(\ref{eq:4-2r-moments})$ is an application of Proposition \ref{prop:difference-moments-estimate}. For the Wigner case the proof follows the same lines, 
by using  
$P(x)= x^6-7 x^4 + \frac{37} 4 x^2$, satisfying
\begin{align*}
\varphi( P( F ) ) =\widehat \kappa_{6}(F)-\widehat\kappa_4(F) +\frac 1 4 \widehat\kappa_2(F) + 3 \, \widehat\kappa^2_3(F) \ge 0 .
\end{align*}
\end{proof}

Now, we are ready to present the main result of the section.

\begin{thm}\label{thm:moments-condition}
Let $\{F_n\}_{n \ge 1}$ be a sequence in the second Wiener chaos such that $\E(F^2_n)=1$ for all $n \ge 1$. 
Then the following asymptotic assertions are equivalent.
\begin{description}
\item[(I)] as $n \to \infty$, $d_{W_2}(F_n, N_1 \times N_2) \to 0$.
\item[(II)] as $n \to \infty$,  sequence $F_n  \to F_\infty \sim N_1 \times N_2$ in distribution.
\item[(III)]  as $n \to \infty$, 
\begin{enumerate}
\item $\mu_4 (F_n) \to 9$.
\item $\mu_{2r}(F_n) \to \big( (2r-1)!!\big)^2$ for some $r \ge 3$.
\end{enumerate}
\end{description}
If moreover $\mu_4(F_n) \ge 9$ for all $n \ge 1$, then $\mu_{2r}(F_n) \ge \big( (2r-1)!!\big)^2$ for $r \ge 3$, and in addition for some constant $C$ (independent of $n$) we obtain
\begin{equation}\label{eq:2r-moment-estimate}
d_{W_2} ( F_n, N_1 \times N_2 ) \le C \, \sqrt{ \mu_{2r}(F_n) - (2r-1)!!^2 }.
\end{equation}
\end{thm}

\begin{proof}
$\bf{(I) \to (II)}$: It is well-known that convergence with respect to probability metric $W_p$, $(p \ge 1)$ is equivalent to the usual weak convergence of measures plus convergence of the first $p$th moments, see \cite{villani-book}. $\bf{(II) \to (III)}$: Let's assume that $F_n$ converges in distribution towards $F_{\infty} \sim N_1 \times N_2$. Then because of hypercontractivity of Wiener chaoses 
(see \cite[Lemma 2.4]{n-p-2w}, or  Lemma \ref{classical:hypercontractivity}), $$\sup_{n \ge 1} \E \vert F^r_n \vert < +\infty, \quad \forall \, r \ge 1.$$ Hence, an application of continuous mapping 
Theorem yields that $\mu_{2r}(F_n) \to (2r-1)!!^2$ for any $r \ge 2$. $\bf{(III) \to (I)}$: note that, since $\sup_{n \ge 1 }\E(F^2_n) < + \infty$, so the sequence
$\{F_n\}_{ n\ge 1}$ is tight, and therefore any subsequence $\{F_{n_k}\}_{ k \ge 1}$ 
contains a further subsequence $\{F_{n_{k_l}}\}_{ l \ge 1}$, and a random variable $F$ such that $F_{n_{k_l}}$ converges in distribution towards $F$ as $l \to \infty$. We need to show that $F \sim F_\infty$. To simplify our argument, we consider two separate cases and also we assume that $\{ F_{n_{k_l}}\}_{ l \ge 1} = \{ F_n \}_{n \ge 1}$.\\
{\it Case (i)}: assume that $\mu_4(F_n) \ge 9$ for all $n \ge 1$. 
Then, item $2$ in ${\bf (III)}$ together with Proposition \ref{prop:4>9} indicates that $F_n$ converges to $F_\infty$ in Wasserstein-2 distance.\\
{\it Case (ii)}: assume that $\mu_4(F_n) < 9$ for all $n \ge 1$. Suppose that there is a subsequence of indexes, such that for each $F_n$ in the subsequence $\E(F_n^2)=1$, $9> \kappa_4(F_n)\longrightarrow \kappa_4(N_1N_2)$, $\mu_{2r}( F_n) \to \mu_{2r}( N_1 N_2)$ for some $r\ge 3$, and we assume that for each $n$  either $\lambda_{n,1}^2 > 1/2$
or $\lambda_{n,-1}^2> 1/2$.  The last condition implies that all odd cumulants have the same sign  i.e.   $\kappa_{2\ell+1}( F_n) \kappa_{2r+1}( F_n)\ge  0$, $\forall  \ell,r\in \N$.

By assumption for such subsequence  $\mu_4( F_n ) \to \mu_4( N_1 N_2)$, and $\mu_{2r}( F_n) \to \mu_{2r}( N_1 N_2)$ for some $r\ge 3$.
We write the cumulant to moment formula into two parts,
\begin{align}\label{align:moment-to-cumulant}
 \mu_{2r}( F_n ) = \sum_{\pi \in \Pi_{2r}^' } \prod_{A\in \pi} \kappa_{|A|}(F_n)+ \sum_{\rho \in \Pi_{2r}^{''} } \underbrace{ \prod_{B\in \rho} \kappa_{|B|}(F_n )   }_{\ge 0 }
\end{align}
where the first sum  is over partitions containing only parts of even size, and the second sum is over partitions containing an even number of odd parts.
Since all odd cumulants have the same sign, the second sum is non-negative, and it vanishes when $F_n$ is a symmetric random variable.
On the other hand, in the first sum, for every even $|A|=2\ell $ with $3 \le \ell \le r$, we have
\begin{align*}
 \kappa_{2\ell }( F_n )&= \frac{ (2\ell -1)! }{ 6 } \kappa_4(F_n) +  \biggl(  \kappa_{2\ell}( F_n) - \frac{ (2\ell-1)! }{ 6 } \kappa_4(F_n) \biggr) 
 & \\ &\ge
 \frac{ (2\ell -1)! }{ 6 } \kappa_4(F_n) +  (2\ell -1)!(\ell-2) \biggl(  \frac{\kappa_4(F_n)} 6 - \kappa_2( F_n) \biggr)
&\end{align*}
and the last inequality is an equality if and only if $F_n= N_1 N_2$. Hence,  by the assumption  $\mu_4( F_n ) \to \mu_4( N_1 N_2)$, for every  $|A|=2\ell$, with $3 \le \ell \le r$, we obtain
\begin{equation}\label{eq:smaller-cumulants}
    \limsup_{n \to \infty} \kappa_{2\ell}(F_n)\ge        \liminf_{n \to \infty} \kappa_{2\ell}(F_n) \ge (2\ell -1)! = \kappa_{2\ell}(N_1N_2).
\end{equation}
We need to show that these are equalities $\forall 3\le \ell \le r$ . Otherwise there would be  some   $3\le \ell \le r$  and $\varepsilon > 0 $ such that
\begin{align*}
 \limsup_{n \to \infty} \kappa_{2\ell}(F_n) = \kappa_{2\ell }(N_1N_2) + \epsilon , 
\end{align*}
which would  lead to
\begin{align*}
  \sum_{\pi \in \Pi_{2r}^' } \prod_{A\in \pi} \kappa_{|A|}(N_1 N_2) &=\mu_{2r}( N_1 N_2) = \lim_{n\to\infty} \mu_{2r}( F_n) \ge  \limsup_{n \to \infty}\sum_{\pi \in \Pi_{2r}^' } \prod_{A\in \pi} \kappa_{|A|}(F_n)
  & \\ &\ge \sum_{\pi \in \Pi_{2r}^' } \prod_{A\in \pi} \biggl(    \kappa_{|A|}(N_1 N_2) + \varepsilon {\bf 1}( |A|=2\ell ) \biggr ) > \mu_{2r}( N_1 N_2) ,
&\end{align*}
with strict inequality.


Hence, the above observation together with relation $(\ref{eq:smaller-cumulants})$  imply that $\kappa_{2\ell}(F_n ) \to  \kappa_{2\ell}(N_1 N_2)$ for every $3 \le \ell \le r$, and therefore, $\mu_{2\ell}(F_n ) \to  \mu_{2\ell}(N_1 N_2)$    $\forall \ell \le r$.
In particular also $\mu_{6}(F_n ) \to  \mu_{6}(N_1 N_2)$, and the conclusion for this subsequence follows by Proposition \ref{prop:4>9}.
It follows also that 
\begin{align*}
 \lim_{n\to\infty} \sum_{\rho \in \Pi_{2r}^{''} } \underbrace{ \prod_{B\in \rho} \kappa_{|B|}(F_n )   }_{\ge 0 }  \longrightarrow 0
\end{align*}
where the sum is over  partitions containing an even number of odd parts, and since all summands are non-negative, this implies that $\kappa_3(F_n)\to 0$.

{\it Case (iii)}:
Otherwise $\lambda_{i,n}^2< 1/2$ $\forall i \in \Z$, and
by item $1$ together with Proposition \ref{prop:4<9} yield that in the representations of $F_n$'s there are exactly two indices $k,l$ such that  $\lambda^2_{n,k}, \lambda^2_{n,l} \to 1/2$ from below, and also all the rest of coefficients tend to $0$ as $n \to \infty$. (note that in principle the indices $k, l$ may depend on $n$. However, this does not affect our argument in below). Since reordering the coefficients does not change the law of $F_n$, we can assume that 
\begin{equation}\label{eq:case-ii}
F_n \sim G_n +H_n := \left( \lambda_{n,1} \frac{ (N^2_1 -1) }{ \sqrt 2 } + \lambda_{n,2}\frac{ (N^2_2 -1)}{\sqrt 2 } \right) + H_n,
\end{equation}
where $\lambda^2_{n,k} \to 1/2$ for $k=1,2$, $G_n$ and $H_n$ are independent, and also $H_n$ belongs to the second Wiener chaos. First note that $1 = \E(F^2_n) = \lambda^2_{n,1} + \lambda^2_{n,2} + \E (H^2_n)$, and so one can infer that $\E(H^2_n) \to 0$ as $n \to \infty$, implying by hypercontractivity argument that as $n \to \infty$, 
\begin{equation}\label{eq:H-moments}
\E(|H|^p_n) \to 0, \quad \forall \, p \ge 2.
\end{equation}
We claim that  $(\lambda_{n,1}\lambda_{n,2} )  \to  -1/2$ and $G_n \stackrel{L^2}{\to} N_1 \times N_2$. By contradiction assume that this is not the case and without loss of generality
there is subsequence with
both $\lambda_{n,1}, \lambda_{n,2}\to 1/\sqrt{2}$  (for a  subsequence with  both $\lambda_{n,1}, \lambda_{n,2}\to -1/\sqrt{2}$ we can flip the sign of all the random variables).   
Then using item $2$, relation $(\ref{eq:H-moments})$, and exploring the independence between $G_n$, and $H_n$,  we get that 
$$(2r-1)!!^2 \leftarrow \mu_{2r}(F_n) \approx \mu_{2r}(G_n) \to \sum_{s=0}^{2r} (-1)^s \frac{ (2r)!}{(2r-s)!} \gneq \biggl(  \frac{2r !}{ r! 2^r} \biggr)^2,$$
which is a contradiction. 
The strict inequality  for $r \ge 3$ follows by the cumulant to moment formula, since for $n$ large enough all odd  cumulants of  $F_n$ in the subsequence  have the same sign.

 Hence, as $n \to \infty$,
$$d_{W_2}(F_n,N_1 \times N_2) ^2\le \E \vert F_n - N_1 N_2 \vert ^2  \le  \E \vert F_n - G_n \vert ^2 + \E(H^2_n) \to 0.$$ 
\end{proof}
\begin{rem}{\rm The proof of Theorem \ref{thm:moments-condition} reveals that under the knowledge of $\mu_4(F_n) \ge 9$ the assumption in item $1$ at ${\bf (III)}$ is immaterial, and it automatically takes place. 
}
\end{rem}
\begin{rem}{ \rm
The quantitative estimate $(\ref{eq:2r-moment-estimate})$ is similar to the main result in \cite{a-m-m-p} for the normal approximation. There it is shown that for a sequence $\{F_n\}_{n \ge 1}$ in an arbitrary Wiener chaos of order $p \ge 2$ with $\E(F^2_n)=1$ the moment inequality $\mu_{2r}(F_n) \ge (2r-1)!!$ take place for every $r \ge 2$. Furthermore, $$d_{\text{TV}}(F_n, \mathscr{N}(0,1)) \le_C \, \sqrt{\mu_{2r}(F) - (2r -1)!!}.$$ Hence, in the normal approximation, to capture the distance in the total variation metric with only one higher even moment it is enough to fix the first even moments (i.e. the second moments). However in the case of the normal product approximation in Wasserstein-2 distance, one needs to perfectly match the first two even moments, i.e. the second and the fourth moments. 
}
\end{rem}

\begin{rem}{ \rm
Thanks to \cite[Theorem 3.1]{n-p-total}, the following quantitative result in total variation distance is also in order.  Let $\{F_n\}_{n \ge 1}$ be a sequence in the second Wiener chaos such that $\E(F^2_n)=1$, and moreover $\E(F^4_n)=9$  for all $n \ge 1$. Then, for all $r \ge 3$,
$$d_{\text{TV}}(F_n, N_1 \times N_2) \le_C \, \sqrt[\leftroot{-3}\uproot{3}10]{\mu_{2r}(F_n) - (2r -1)!!^2}.$$
}
\end{rem}

\begin{cor}\label{thm:moments-condition:non:commutative}
Let $\{F_n\}_{n \ge 1}$ be a sequence of non-commutative random variables in the second Wigner chaos such that $\varphi(F^2_n)=1$ for all $n \ge 1$. Assume that $F_\infty$ distributed as normalized tetilla distribution. 
Then, as $n \to \infty$, the following asymptotic assertions are equivalent.
\begin{description}
\item[(I)] $d_{W_2}(F_n,F_\infty) \to 0$.
\item[(II)] sequence $F_n  \to F_\infty$ is distribution.
\item[(III)]  as $n \to \infty$, 
\begin{enumerate}
\item $\varphi (F_n^4) \to 5/2$.
\item $\varphi(F_n^{2r} ) \to  \varphi(F_{\infty}^{2r} )     $ for some $r \ge 3$.
\end{enumerate}
\end{description}
If moreover $\varphi(F^4_n) \ge 5/2$ for all $n \ge 1$, then $\varphi(F^{2r}_n) \ge \varphi(F^{2r}_{\infty}) $ for all $r \ge 3$, and in addition for some constant $C$ (independent of $n$) we obtain
\begin{equation}\label{eq:2r-moment-estimate:free}
d_{W_2} ( F_n, F_{\infty} ) \le_C \, \sqrt{ \varphi(F^{2r}_n) -  \varphi(F^{2r}_{\infty} )  }.
\end{equation}
\end{cor}
\begin{proof}As in  Theorem \ref{thm:moments-condition}. 
\end{proof}

\subsection{Asymptotic result for the coupled sequence}
The materials of this section are inspired from the proof of Theorem \ref{thm:moments-condition}. Our new setup is the following. We consider now the convergence in Wasserstein-2 distances
for the laws of a sequence of non commutative random variables $(F_n:n\in \N)\in L^2$  with representation
 \begin{align} \label{representation:coupling}
   F_n = \sum_{z\in \Z} \lambda_{n,z} X_z, 
\end{align}
where $(X_z:z\in \Z)\subseteq L^2({\mathcal A},\varphi)$ is a sequence of identically distributed free random variables with $\varphi(X_1)=0$, $\varphi(X_1^2)=1$, and  for each $n\ge 0$ the sequence  $\lambda_n=(\lambda_{n,z}: z\in \Z) \in \ell^2 (\Z)$ is  ordered as in \eqref{eigenvalue:ordering}, i.e. we consider the situation that $\lambda_{n,0}=0$, and moreover
\begin{align} \label{ordering}
-\Vert F \Vert_{L^2} \le \lambda_{n,-1 } \le \lambda_{n,-2} \le \dots \le \lambda_{n,-z}  \le \dots \le 0 \le \dots \le \lambda_{n,z} \le \dots \le \lambda_{n,2} \le \lambda_{n,1} \le \Vert F \Vert_{L^2}.
\end{align} 
In the commutative case 
 $(X_z:z\in \Z)\subseteq  L^2(\Omega,{\mathcal F},P)$ is a sequence of classically  independent and identically distributed random variables with $\E(X_1)=0$, $\E(X_1^2)=1$. In the special case where  $X_1\stackrel{law}{=} -X_1$, we also assume  without loss of generality that $\lambda_{n,z}=0$, $\forall z\le 0\le n$. First, we need the following enlightening lemma telling us that the best permutation in the definition of $d_\sigma$- distance in \cite[relation $(2.1)$, page 4]{a-a-p-s} is given by ordering $(\ref{ordering})$.


\begin{lem} \label{lem:best-permutation}
Under the above setting, for every two (non) commutative square integrable random variables $F_1$ and $F_2$ following representation $(\ref{representation:coupling})$, and $(\ref{ordering})$, we have
\begin{align} \label{wasserstein_2:permute}
  d_{W_2}( F_1, F_2)^2 \le \parallel \lambda_1- \lambda_2 \parallel^2_{\ell^2(\Z) }\le \parallel \lambda_1 \circ \pi - \lambda_2 \parallel^2_{\ell^2(\Z) }
\end{align}
for any bijection  $\pi: \Z\to\Z$.
\end{lem}
\begin{proof} The representation    \eqref{representation:coupling} with the same free sequence $(X_z:z \in Z)$  gives a coupling of the non-commutative random variables $(F_n:n \in \N)$,
and by using freeness we obtain
 \begin{align*}
 &  \varphi( (F_1-F_2)^2 ) = \sum_{z, z'\in \Z }( \lambda_{1,z}- \lambda_{2,z} )   ( \lambda_{1,z'}- \lambda_{2,z'} )     \varphi\bigl(  X_z X_{z'} \bigr)       
 & \\ &= \sum_{z\in \Z}  ( \lambda_{1,z}- \lambda_{2,z} )^2 \varphi\bigl( X_z^2\bigr)=  \parallel \lambda_{1} - \lambda_{2} \parallel_{\ell^2(\Z)}^2.
&\end{align*}
The inequality  in \ref{wasserstein_2:permute} is equivalent to
\begin{align}\label{hardy:ineq}
  \sum_{z\in \Z} \lambda_{1,z} \lambda_{2,\pi(z)} \le\sum_{	z\in \Z} \lambda_{1,z} \lambda_{2,z},  \quad  \forall \mbox{ bijection } \pi : \Z\to\Z .
\end{align}
For finite vectors this is known as {\it rearrangement} inequality  
\cite[Thm. 368]{hardy}.
For infinite sequences, it is clear that non-negative (non-positive)  $\lambda_{1}$ coordinates should be matched with 
non-negative (non-positive) $\lambda_{2}$ coordinates, and  \eqref{hardy:ineq}
 follows from
 the  Hardy-Littlewood inequality 
for  function rearrangements \cite[Thm. 378]{hardy}, 
applied separately to the functions piecewise constant on unit intervals
corresponding to the non-negative (non-positive) subsequences.

\end{proof}

We next study two possible situations: 
\begin{enumerate}
 \item[\rm  (C):]  $( F_n:n \in \N) \subset \mathscr{H}_2$, the second Wiener chaos. Equivalently, $F_n$ have representation   \eqref{representation:coupling} 
 and there is a sequence   $(N_z: z\in \Z)$ of independent standard  Gaussian random variables on a classical probability space $(\Omega,{\mathcal F},\P)$
 such that  $X_{z}= (N_z^2 -1)/ \sqrt{2} $, $\forall z\in \Z$.  
 \item[\rm (N-C):] $( F_n:n \in \N) \subset \mathscr{H}_2^S$, the second Wigner chaos. Equivalently, $F_n$ have representation   \eqref{representation:coupling} and there is 
  a sequence  
 $(S_z: z\in \Z)$ of  freely independent normalized semicircular random variables  defined on a non-commutative probability space  $({\mathcal A},\varphi)$
 such that $X_{z}= (S_z^2 -1)$, $\forall z\in \Z$.
\end{enumerate}


The main result in this section is the following.
\begin{thm} 
 In both commutative and non-commutative settings, under the above settings, and assumptions {\rm (C)} and {\rm (N-C)}, and the second moment condition $\E(F^2_n)=1 (\varphi(F^2_n)=1)$ for all $n \ge 1$, with the target random variable 
\begin{align*}
F_{\infty} := \frac{ X_1- X_{-1} } {\sqrt  2} \stackrel{law}{=} \begin{cases} ( S_1 S_2 + S_2 S_1)/\sqrt{ 2}  &  , \mbox{ \rm(N-C) }\\ 
                                                                                                                            N_1 \times N_2 &  , \mbox{ \rm(C) }
                                                                                                                           \end{cases}
                                                                                                                             \end{align*}
                                                                                                                             as $n \to \infty$, the following asymptotic assertions are equivalent.
\begin{description}
\item[(I)]   $\forall p\ge 1$,   $F_n \stackrel{L^p}{\to}  F_{\infty} $.
\item[(II)]    $F_n \stackrel{L^2}{\to}  F_{\infty} $.
\item[(III)]  $\parallel \lambda_n - \lambda_{\infty}  \parallel_{\ell^2(\Z) } =\bigl  (2+ \sqrt 2 \lambda_{n,-1}- \sqrt 2 \lambda_{n,1} \bigr) \to 0$, with limiting sequence 
  $\lambda_{\infty,z}=0$ for $|z|\ne 1$, and $\lambda_{\infty,\pm 1}=\pm 1 /\sqrt{2}$.
\item[(IV)]  $d_{W_2}( F_n ,  F_{\infty})\to  0$.
 \item[(V)]  $ F_n \stackrel{law}{\to} F_{\infty}$
\item[(VI)] for some $r\ge 3$;
 \begin{align*}
  \mu_{4}( F_n) \to \mu_{4} (F_{\infty} ), \quad \text{and} \quad  \mu_{2r}( F_n) \to \mu_{2r} (F_{\infty} )\\
\Big( \varphi(F^4_n) \to 2.5, \quad \text{and} \quad \varphi(F^{2r}_n) \to \varphi(F^{2r}_\infty)\Big).
\end{align*}
\end{description}
\end{thm}
\begin{proof}
We consider the commutative case. The chain of implications ${\bf (I)} \to {\bf (II)} \to {\bf (III)} \to {\bf (IV)} \to {\bf (V)} \to {\bf (VI)}$ is straightforward. Note that under $(\eqref{representation:coupling})$, $(\ref{ordering})$, and {\rm (C)}, using Lemma \ref{lem:best-permutation}, and \cite[Theorem 2.3]{a-a-p-s} together with Proposition \ref{prop:4>9}, one can infer that
\begin{equation*}
\begin{split}
 d_{W_2}(F_n,F_\infty) &\le \E\bigl( (F_n -F_\infty)^2 \bigr)
 = \parallel \lambda_{n} - \lambda_{\infty} \parallel_{\ell^2(\Z)}^2\\
 & \le_C \sqrt{\left( \mu_6(F_n) - \mu_6(F_\infty) \right) - 55 \left( \mu_4(F_n) - \mu_4 (F_\infty) \right)}.
 \end{split}
\end{equation*}
This implies implication ${\bf (VI)} \to {\bf (II)}$. Lastly, ${\bf (II)} \to {\bf (I)}$ is just the hypercontractivity of the second Wiener chaos, see Lemma \ref{classical:hypercontractivity}.
\end{proof}

\section{Conjecture}\label{sec:conjecture}
The final message of our study is the following. For a normalized sequence $\{F_n\}_{n \ge 1}$ of classical random variables  
in the second Wiener chaos (non-commutative random variables in the second Wigner chaos, respectively)
the convergences of the {\bf fourth} and {\bf another} higher even moments to the corresponding even moments of 
$N_1 \times N_2$  (of the normalized tetilla law, respectively)
are necessary and sufficient conditions for convergence in distribution.

It turns out that in our analysis the fourth moment plays a substantial role.  At present it is unclear how one can replace the convergence of the fourth moments in Theorems \ref{thm:moments-condition} and \ref{tetilla:moment:characterization} 
by convergence of another higher even moments. Nevertheless we believe that such replacement would be possible, and we leave it in the shape of a conjecture in below. 

\begin{conj}
Let $\{F_n\}_{n \ge 1}$ be a sequence of random variables in a fixed Wiener (Wigner) chaos of order $p\ge 2$ such that $\mu_2(F_n)=1\, (\varphi(F^2_n)=1)$, for all $n \ge 1$. 
Assume that the target distribution $F_{\infty} \sim N_1 \times N_2$ (normalized tetilla law). Then the following asymptotic assertions are equivalent;
\begin{description}
\item[(I)]  $F_n  \to F_\infty$ in distribution  as $n \to \infty$.
\item[(II)]  as $n \to \infty$, for some $3 \le r \neq s $,
\begin{enumerate}
\item $\mu_{2s} (F_n) \to \mu_{2s}(F_{\infty}), \quad \left(\varphi(F^{2s}_n) \to \varphi(F^{2s}_\infty) \right)$.  
\item $\mu_{2r}(F_n)   \to \mu_{2r}(F_{\infty}), \quad \left(\varphi(F^{2r}_n) \to \varphi(F^{2r}_\infty) \right)$.  
\end{enumerate}
\end{description}
\end{conj}

\
(E. Azmoodeh) \textsc{Faculty of Mathematics, Ruhr University Bochum, Germany}, \emph{E-mail address}: {\tt ehsan.azmoodeh@rub.de}

(D. Gasbarra) \textsc{Department of Mathematics and Statistics, University of Helsinki, Finland}, \emph{E-mail address}:  {\tt dario.gasbarra@helsinki.fi }
\

\end{document}